\documentclass[12pt,a4paper]{article}
\usepackage[utf8]{inputenc}

\usepackage{amsmath,amssymb}
\usepackage{amsthm}
\usepackage{enumitem}
\usepackage{appendix}
\usepackage{theoremref}

\usepackage{graphicx}

\usepackage[usenames]{color}

\usepackage{txfonts}

\usepackage[margin=1in]{geometry}
\usepackage{amsmath}
\usepackage{mathrsfs}
\usepackage{amsfonts}
\usepackage{cite}
\usepackage[mathscr]{euscript}
\usepackage{amssymb}
\usepackage{appendix}
\usepackage{tikz-cd}
\usepackage{empheq}
\usepackage{mathtools}
\usepackage{theoremref}
\usepackage{hyperref}
\usepackage{comment}
\usepackage{mathabx}
\usepackage{dsfont}
\usepackage{multicol}
\usepackage{graphicx}
\usepackage{float}
\usepackage{enumitem}
\usepackage{subcaption}
\usepackage{array}
\usepackage{mhsetup}
\usepackage[misc]{ifsym}
\usepackage{datetime}
\usepackage[in]{fullpage}
\usepackage{algorithm}
\usepackage{algorithmic}
\usepackage{authblk}

\newcommand{\ben}{\begin{eqnarray*}}
\newcommand{\enn}{\end{eqnarray*}}

\newtheorem{thm}{Theorem}[section]
\newtheorem{lem}[thm]{Lemma}
\newtheorem{defn}[thm]{Definition}
\newtheorem{prob}[thm]{Problem}
\newtheorem{prop}[thm]{Proposition}

\newtheorem{rem}[thm]{Remark}

\newtheorem{ass}[thm]{Assumption}
\definecolor{dhcol}{rgb}{0,0.5,0}

\definecolor{cecol}{rgb}{0,0,0.5}

\newcommand{\mres}{%
	\,\raisebox{-.127ex}{\reflectbox{\rotatebox[origin=br]{-90}{$\lnot$}}}\,%
}

\title{Maximal Codimension Collisions and Invariant Measures for Hard Spheres on a Line}
\author{Mark Wilkinson\footnote{Department of Physics and Mathematics, New Hall Block, Nottingham Trent University, Nottingham, United Kingdom (\href{mark.wilkinson@ntu.ac.uk}{mark.wilkinson@ntu.ac.uk}).}}

\date{}

\begin{document}
\maketitle
\begin{abstract}
\noindent For any $N\geq 3$, we study invariant measures of the dynamics of $N$ hard spheres whose centres are constrained to lie on a line. In particular, we study the invariant submanifold $\mathcal{M}$ of the tangent bundle of the hard sphere billiard table comprising initial data that lead to the simultaneous collision of all $N$ hard spheres. Firstly, we obtain a characterisation of those continuously-differentiable $N$-body scattering maps which generate a billiard dynamics on $\mathcal{M}$ admitting a canonical weighted Hausdorff measure on $\mathcal{M}$ (that we term the \emph{Liouville measure on $\mathcal{M}$}) as an invariant measure. We do this by deriving a second boundary-value problem for a fully nonlinear PDE that all such scattering maps satisfy by necessity. Secondly, by solving a family of functional equations, we find sufficient conditions on measures which are absolutely continuous with respect to the Hausdorff measure in order that they be invariant for billiard flows that conserve momentum and energy. Finally, we show that the unique momentum- and energy-conserving \emph{linear} $N$-body scattering map yields a billiard dynamics which admits the Liouville measure on $\mathcal{M}$ as an invariant measure.
\end{abstract}
\setcounter{tocdepth}{2}
\tableofcontents
\section{Introduction}\label{intropaper}
In the theory of billiard dynamics on billiard tables which possess \emph{corners}, it is often the case that trajectories are only constructed for those initial data that do not lead to the corners themselves. Indeed, to quote Smillie (\cite{smillie}, pp.360-361), if ``the trajectory hits a corner then there may be no good physical principal [sic] which selects one particular continuation. In this case the continuation of the trajectory is undefined.'' We also refer the reader to the works (\cite{pnueli2021structure}, \cite{de2014expansion}, \cite{kenyon2000billiards}, \cite{gutkin2012billiard}, \cite{chernov2006chaotic}) in which corner dynamics is considered. As a result, the associated billiard flow operators are undefined on a topologically- or measure theoretically-negligible subset of the tangent bundle of the table (or a sphere bundle over the table, if one works in a phase space with unit velocities). If one is to try to extend the domain of definition of these billiard flow operators to the whole tangent bundle of the table, thereby extending the definition of a trajectory past those times it hits a corner, one has to deal with the generic \emph{non-uniqueness} of trajectory extension past this time of collision. In order to decide what constitutes a `good' or `natural' extension of the dynamics, we shall consider \emph{analytical} properties of trajectories \emph{in-the-large} on phase space, rather than any physical principle which is otherwise lacking.

One well-studied billiard table with corners, of interest to those both in the pure and applied mathematics communities (see, for example, Bodineau, Gallagher, Saint-Raymond and Simonella \cite{laurespheres}, or Viazovska \cite{viazovska2017sphere} in the case the spheres are 8-dimensional), is the set of all admissible centres of mass for a collection of $N$ hard spheres evolving in space. Indeed, we denote by $\mathcal{B}\subset\mathbb{R}^{3N}$ the \emph{hard sphere billiard table} given by
\begin{equation}\label{allthetable}
\mathcal{B}:=\left\{
X=(x_{1}, ..., x_{N})\in\mathbb{R}^{3N}\,:\,|x_{i}-x_{j}|\geq \varepsilon\hspace{2mm}\text{for}\hspace{2mm}i\neq j
\right\},
\end{equation}
with $x_{i}=(x_{i, 1}, x_{i, 2}, x_{i, 3})\in\mathbb{R}^{3}$ modelling the centre of mass of a sphere labelled by $i\in\{1, ..., N\}$, and where $\varepsilon>0$ denotes the radius of each of the congruent spheres. Corners in the boundary $\partial\mathcal{B}$ of the billiard table correspond to configurations of the centres of mass which characterise simultaneous collision of $M$ of these spheres, where $3\leq M\leq N$. In the particular case of momentum- and energy-conserving hard sphere dynamics on $\mathbb{R}^{3}$ or subsets thereof, with which the present article is largely concerned, the matter of constructing global-in-time billiard trajectories which start from initial data in these high codimension components of the tangent bundle $T\mathcal{B}$ of the billiard table amounts in the first instance to a problem in \emph{algebraic geometry}. Indeed, one needs to construct a \textbf{scattering map} that has range in a non-trivial semi-algebraic subvariety of $\mathbb{R}^{3N}$, which leads to the aforementioned absence of uniqueness in trajectory extension. 

In this article, we study the connection between the analytical properties of such scattering maps and the analytical properties of the billiard flows on phase space that they generate. In particular, in the reduced setting of linear hard sphere dynamics on the tangent bundle $T\mathcal{B}$ of the hard sphere table $\mathcal{B}$, we use the concept of \textbf{invariant measure} (and the PDE involving their densities) to formulate a condition as to what constitutes a `good' extension of billiard trajectory past corners of the table $\mathcal{B}$. By establishing a natural generalisation of \emph{Liouville measure} on the subvariety of $T\mathcal{B}$ on which we work, we prove in the case of momentum- and energy-conserving billiard dynamics that there is a \emph{unique} linear scattering map generating a billiard flow on $T\mathcal{B}$ that admits this Liouville measure as an invariant measure.

As the construction of scattering maps involves rudimentary (semi-) algebraic geometry, we now set out the problem of their construction.
\subsection{A Problem in Algebraic Geometry}
To begin, we set up some notation that will be helpful in defining the algebraic varieties with which we work in the sequel. For given $Z=(X, V)\in T\mathcal{B}$, we recall from the above definition \eqref{allthetable} of the billiard table that $X\in\mathcal{B}$ denotes the vector of centre-of-mass data given component-wise by
\begin{equation*}
X_{i}:=x_{\lceil i/3 \rceil, [i-1]_{3}+1}     
\end{equation*}
for $i\in\{1, ..., 3N\}$ with $[i]_{3}:=i\,(\text{mod}\,3)$, while $V\in T_{X}\mathcal{B}\subset\mathbb{R}^{3N}$ analogously denotes the velocity vector given component-wise by
\begin{equation*}
V_{i}:=v_{\lceil i/3 \rceil, [i-1]_{3}+1}    
\end{equation*}
where $v_{i}=(v_{i, 1}, v_{i, 2}, v_{i, 3})\in\mathbb{R}^{3}$ denotes the velocity of the centre of mass of the hard sphere labelled by $i$. In turn, for $Z\in T\partial\mathcal{B}$, let $\mathcal{Y}_{Z}\subset\mathbb{R}^{3N}$ denote the algebraic variety determined by the following 7 algebraic equations in the variable $Y=(Y_{1}, ..., Y_{3N})\in\mathbb{R}^{3N}$: 
\begin{equation}\label{linmomeqs}
\left\{
\begin{array}{l}
\displaystyle \sum_{i=1}^{N}Y_{3i-2}=\sum_{i=1}^{N}V_{3i-2}, \vspace{2mm}\\
\displaystyle \sum_{i=1}^{N}Y_{3i-1}=\sum_{i=1}^{N}V_{3i-1}, \vspace{2mm}\\
\displaystyle \sum_{i=1}^{N}Y_{3i}=\sum_{i=1}^{N}V_{3i}, \vspace{2mm}
\end{array}
\right.
\end{equation}
\begin{equation}\label{angmomeqs}
\left\{
\begin{array}{l}
\displaystyle \sum_{i=1}^{N}X_{3i-1}Y_{3i}-X_{3i}Y_{3i-1}=\sum_{i=1}^{N}X_{3i-1}V_{3i}-X_{3i}V_{3i-1}, \vspace{2mm}\\
\displaystyle \sum_{i=1}^{N}X_{3i}Y_{3i-2}-X_{3i-2}Y_{3i}=\sum_{i=1}^{N}X_{3i}V_{3i-2}-X_{3i-2}V_{3i}, \vspace{2mm}\\
\displaystyle \sum_{i=1}^{N}X_{3i-2}Y_{3i-1}-X_{3i-1}Y_{3i-2}=\sum_{i=1}^{N}X_{3i-2}V_{3i-1}-X_{3i-1}V_{3i-2},
\end{array}
\right.
\end{equation}
and
\begin{equation}\label{keeqs}
\sum_{i=1}^{N}Y_{i}^{2}=\sum_{i=1}^{N}V_{i}^{2}. 
\end{equation}
The first three equations \eqref{linmomeqs} correspond to the conservation of linear momentum, the second three \eqref{angmomeqs} to the conservation of angular momentum (measured from the origin), and the final equation \eqref{keeqs} to the conservation of kinetic energy. The structure of the variety $\mathcal{Y}_{Z}$ determined by these algebraic equations can be understood, for instance, by computation of a Gr\"{o}bner basis of the above set of polynomials through Buchberger's algorithm: see Hassett (\cite{hassett2007introduction}, Chapter 2) for a clear description thereof. 

Additionally, if one stipulates that billiard trajectories $t\mapsto X(t)$ are both left- and right-differentiable on $\mathbb{R}$, the above system is supplemented with the system of $X$-dependent \emph{semi-algebraic} conditions given by
\begin{equation}\label{semialgebraicconditions}
\sum_{k=1}^{3}\left(X_{3(i-1)+k}-X_{3(j-1)+k}\right)(Y_{3(i-1)+k}-Y_{3(j-1)+k})\leq 0
\end{equation}
for $(i, j)\in\Gamma_{X}$, where
\begin{equation*}
\Gamma_{X}:=\left\{(i, j)\in\{1, ..., N\}^{2}\setminus \Delta_{N}\,:\,\sum_{k=1}^{3}(X_{3(i-1)+k}-X_{3(j-1)+k})^{2}=\varepsilon^{2}\right\}.    
\end{equation*}
These semi-algebraic conditions correspond to the set of all admissible \emph{pre-collisional} velocities, given the spatial configuration of the spheres $X\in\partial\mathcal{B}$. For given $X\in\partial\mathcal{B}$, we write $\mathcal{V}_{X}^{-}\subset\mathbb{R}^{3N}$ to denote the set of all $Y\in\mathbb{R}^{3N}$ satisfying condition \eqref{semialgebraicconditions} for all $(i, j)\in\Gamma_{X}$. The analogous set of \emph{post-collisional} velocities is denoted by $\mathcal{V}_{X}^{+}:=-\mathcal{V}_{X}^{-}$. 

Given spatial data $X\in\partial\mathcal{B}$ in a corner of the table, as well as a velocity vector $V\in\mathcal{V}_{X}^{-}$ which is pre-collisional with respect to $X$, when $N\gg 1$ it is not trivial to determine the structure of the algebraic variety $\mathcal{Y}_{Z}\cap\mathcal{V}_{X}^{+}\subset\mathbb{R}^{3N}$ for $Z=(X, V)\in T\partial\mathcal{B}$. We invite the reader to consult Coste \cite{10.1007/3-540-51683-2_21} or Basu, Pollack and Roy (\cite{basu2006algorithms}, Chapter 14) for techniques which allow one to determine the structure of the variety determined by the above semi-algebraic system \eqref{linmomeqs}--\eqref{semialgebraicconditions}.
\subsection{Scattering Maps and the Flows they Generate}
As alluded to above, it is not obvious\footnote{Nevertheless, the reader can check in the relatively-simple case $N=3$ that the set $\mathcal{V}_{X}^{-}$ is a non-$\mathscr{L}_{9}$-null subset of $\mathbb{R}^{9}$ for \emph{all} $X\in \partial\mathcal{B}$, with $\mathscr{L}_{9}$ denoting the Lebesgue measure on $\mathbb{R}^{9}$} \emph{a priori} that $\mathcal{V}_{X}^{-}\neq \varnothing$ for all $X\in\partial\mathcal{B}$ and all $N\geq 3$. Moving forward, we define the set $\mathcal{G}\subset\partial\mathcal{B}$ to be
\begin{equation*}
\mathcal{G}:=\left\{X\in\partial\mathcal{B}\,:\,\mathcal{V}_{X}^{-}\neq \varnothing\right\}.    
\end{equation*}
With the above notation in place, one can consider the problem of the construction of maps $\Sigma$ with
\begin{equation*}
\Sigma:\bigsqcup_{X\in\mathcal{G}}\mathcal{V}_{X}^{-}\rightarrow\bigsqcup_{X\in\mathcal{G}}\mathcal{V}_{X}^{+}    
\end{equation*}
subject to the range condition
\begin{equation*}
(\Pi_{2}\circ\Sigma)(Z)\in\mathcal{Y}_{Z},    
\end{equation*}
where $\Pi_{2}:T\mathcal{B}\rightarrow\mathbb{R}^{3N}$ denotes the canonical velocity projection map: see the notation Section \ref{notation} below. Thus, the map $\Sigma$ maps pre-collisional velocity data to post-collisional velocity data in such a way that linear momentum, angular momentum, and kinetic energy are conserved. We call any such map $\Sigma$ a momentum- and energy-conserving \textbf{scattering map}. One may use any such scattering map $\Sigma$ to construct a billiard flow $\{T^{t}_{\Sigma}\}_{t\in\mathbb{R}}$ on a subset of $T\mathcal{B}$. The reader may consult Ballard \cite{ballard2000dynamics} for further details on the construction of dynamics for hard sphere systems. 

There are generically uncountably-many momentum- and energy-conserving scattering maps $\Sigma$. As such, there are uncountably-many associated billiard flows $\{T^{t}_{\Sigma}\}_{t\in\mathbb{R}}$ on $T\mathcal{B}$. As binary collisions of hard spheres which conserve momentum and energy are uniquely determined -- or, in the language of this Section, it holds that $\mathcal{Y}_{Z}\cap\mathcal{V}_{X}^{+}$ is a singleton for any $Z\in T\mathcal{B}$ such that $\Gamma_{X}$ is itself a singleton -- the difference between any two flow operators $T^{t}_{\Sigma_{1}}$ and $T^{t}_{\Sigma_{2}}$ generated by momentum- and energy-conserving scattering maps $\Sigma_{1}$ and $\Sigma_{2}$ is seen only on topologically- or measure-theoretically negligible subsets of $T\mathcal{B}$. Nevertheless, if one is interested in the \emph{global} structure of billiard dynamics on $T\mathcal{B}$, it is important to derive natural analytical criteria on scattering maps $\Sigma$ which result in `good' properties of the associated flow operators $T^{t}_{\Sigma}$. The point of view we take in this article arises from the study of \emph{invariant measures} on $T\mathcal{B}$ other than the well-known \emph{Liouville measure} $\Lambda:=\mathscr{L}_{6N}\mres T\mathcal{B}$: see Chernov and Markarian (\cite{chernov2006chaotic}, Chapter 2), for instance. We focus on the following problem, which we choose to state simply and without complete rigour for the moment:
\begin{prob}\label{looseprob}
Let $N\geq 3$. Suppose $\mathcal{M}\subset T\mathcal{B}$ is an invariant manifold for any billiard flow on $T\mathcal{B}$. For a given measure $\mu$ on $\mathcal{M}$ which is absolutely continuous with respect to the Hausdorff measure on $\mathcal{M}$, find all $\Sigma$ such that
\begin{equation*}
T^{t}_{\Sigma}\#\mu=\mu    
\end{equation*}
for all $t\in\mathbb{R}$.
\end{prob}
In order to make this problem more tractable to study in the present article, we reduce our study of free motion of hard spheres on $\mathbb{R}^{3}$ to the case when the centres of mass of the spheres are constrained to lie on a line. The associated Problem \ref{looseprob} is still not trivial in this case. As we shall show in the sequel, when $\Sigma$ admits some smoothness properties, solving this Problem results in the identification of a second boundary-value problem for a fully-nonlinear PDE characterising all desired scattering maps $\Sigma$: see Trudinger and Wang \cite{trudinger2009second} for the definition of such boundary-value problems. This simplification on the structure of the billiard table results in the well-studied \textbf{hard rod} or so-called \textbf{one-dimensional hard sphere} model.
\subsection{The Scattering Varieties in the `One-dimensional' Case}
The structure of the algebraic system introduced above becomes more straightforward when we assume that the motion of the $N$ spheres is constrained to lie on a line in $\mathbb{R}^{3}$. More precisely, in the particular case that the spatial datum $X\in\partial\mathcal{B}$ is of the form
\begin{equation}\label{simpleboundarypoint}
X=(d_{1}a, ..., d_{N}a)\in\partial\mathcal{B}\subset\mathbb{R}^{3N}    
\end{equation}
for some $a\in\mathbb{S}^{2}$ and $d_{i}\in\mathbb{R}$ satisfying $d_{i+1}-d_{i}=\varepsilon$ with $i\in\{1, ..., N-1\}$, and the velocity vector is taken to be
\begin{equation*}
V=(y_{1}a, ..., y_{N}a)\in\mathbb{R}^{3N}    
\end{equation*}
for some $y_{i}\in\mathbb{R}$, the system of algebraic equations \eqref{linmomeqs}--\eqref{keeqs} reduces to the considerably simpler system in the unknown $Y=(y_{1}, ..., y_{N})\in\mathbb{R}^{N}$ given by 
\begin{equation}\label{easyalgebra}
\left\{
\begin{array}{l}
\displaystyle \sum_{i=1}^{N}y_{i}=\sum_{i=1}^{N}v_{i},\vspace{2mm}\\
\displaystyle \sum_{i=1}^{N}y_{i}^{2}=\sum_{i=1}^{N}v_{i}^{2},
\end{array}
\right.
\end{equation}
while the set of semi-algebraic conditions \eqref{semialgebraicconditions} reduces to
\begin{equation}\label{easyineq}
y_{i+1}-y_{i}\geq 0   
\end{equation}
for $i\in\{1, ..., N\}$. Notably, this system does not depend on the choice of $X\in\partial\mathcal{B}$ of the form \eqref{simpleboundarypoint}. We denote the set of solutions $Y$ of \eqref{easyalgebra} by $\mathcal{Y}_{V}$, and the non-empty set of points $Y$ satisfying the inequalities \eqref{easyineq} by $\mathcal{V}^{-}$. If $\mathcal{V}^{+}:=-\mathcal{V}^{-}$, then it holds that the collection of all (what we term) \textbf{$N$-body scattering maps} $\sigma:\mathcal{V}^{-}\rightarrow\mathcal{V}^{+}$ satisfying the range condition 
\begin{equation*}
\sigma(V)\in\mathcal{Y}_{V}    
\end{equation*}
for all $V\in\mathcal{V}^{-}$ is uncountable, whence the number of ways to extend a billiard trajectory which meets the corner point $X\in\partial\mathcal{B}$ is uncountable. In this article, we identify analytical conditions on the scattering map $\sigma$ under which the means to extend such a billiard trajectory is unique: see the statements of our main Theorems \ref{firstmainresult} and \ref{uniquelinear} below.
\subsection{Reduction to the Fundamental Billiard Table $\mathcal{Q}$}
In this Section, we derive what we term the \emph{fundamental billiard table} $\mathcal{Q}\subset\mathbb{R}^{N}$ that models the `one-dimensional' motion of $N$ hard spheres that will be of central focus in this article. We begin with the following definition.
\begin{defn}[Linear Configurations of Hard Spheres in $\mathbb{R}^{3}$]
For any $N\geq 3$, we define the set of all linear configurations of $N$ hard spheres $\mathcal{L}\subset\mathcal{B}$ by
\begin{equation*}
\mathcal{L}:=\left\{
\Xi=(\xi_{1}, ..., \xi_{N})\in\mathcal{B}\,:\,\{\xi_{i}\}_{i=1}^{N}\hspace{2mm}\text{are collinear in}\hspace{1mm}\mathbb{R}^{3}
\right\}.
\end{equation*}
\end{defn}
Let us now write $\mathcal{L}$ in such a way it reveals the role of what we term the \emph{fundamental table} $\mathcal{Q}$ in its definition. For any $N\geq 3$ and for any given vectors $a, b\in\mathbb{R}^{3}$, we define the set $\mathcal{L}(a, b)\subset\mathcal{B}$ by
\begin{equation*}
\mathcal{L}(a, b):=\left\{
\Xi\in\mathcal{B}\,:\,\Xi=M(a)s+M(b)\mathbf{1}\hspace{2mm}\text{for some}\hspace{1mm}s\in\mathbb{R}
\right\},
\end{equation*}
where $M:\mathbb{R}^{3}\rightarrow\mathbb{R}^{3N\times N}$ is defined pointwise by
\begin{equation*}
M(a):=\sum_{i=1}^{3N}\sum_{j=1}^{N}(a_{[i-1]_{3}+1}\delta_{\lceil i/3\rceil, j})e_{i}^{3N}\otimes e_{j}^{N}    
\end{equation*}
and $\mathbf{1}\in\mathbb{R}^{3N}$ denotes the vector all of whose entries are 1, namely $\mathbf{1}:=\sum_{i=1}^{3N}e_{i}^{3N}$. Simply put, the set $\mathcal{L}(a, b)$ models the set of all admissible centres of $N$ hard spheres which lie on a line in $\mathbb{R}^{3}$ that is collinear with $a$ and that passes through $b$. In turn, the set $\mathcal{L}\subset\mathcal{B}$ may be written as
\begin{equation*}
\mathcal{L}=\bigcup_{a\in\mathbb{R}^{3}\setminus\{0\}}\bigcup_{b\in\mathbb{R}^{3}}\mathcal{L}(a, b).    
\end{equation*}
We also introduce the `one-dimensional' billiard table $\mathcal{P}\subset\mathbb{R}^{N}$ to be
\begin{equation*}
\mathcal{P}:=\left\{
X=(x_{1}, ..., x_{N})\in\mathbb{R}^{N}\,:\,|x_{i}-x_{j}|\geq \varepsilon\hspace{2mm}\text{for}\hspace{1mm}i\neq j
\right\},
\end{equation*}
noting that $x_{i}$ now represents an element of $\mathbb{R}$, rather than $\mathbb{R}^{3}$. In this notation, it holds that
\begin{equation*}
\mathcal{L}(a, b)=M(a)\mathcal{P}+M(b)\mathbf{1}.    
\end{equation*}
We now approach the definition of one of the basic objects of study in this article.
\begin{defn}[The Fundamental Billiard Table $\mathcal{Q}$]
For any $N\geq 3$, we define the \textbf{fundamental billiard table} $\mathcal{Q}\subset\mathbb{R}^{N}$ to be
\begin{equation*}
\mathcal{Q}:=\left\{
X\in\mathbb{R}^{N}\,:\,x_{i}\leq x_{i+1}\hspace{2mm}\text{for}\hspace{1mm}i\in\{1, ..., N-1\}
\right\}.
\end{equation*}
\end{defn}
Using the notation established above, we may write
\begin{equation*}
\mathcal{P}=\bigcup_{\pi\in\mathfrak{s}(N)}(\pi\circ S_{\varepsilon})\mathcal{Q},
\end{equation*}
with $\mathfrak{s}(N)$ denoting the symmetric group on $N$ letters, and the shift map $S_{\varepsilon}:\mathcal{Q}\rightarrow\mathcal{P}$ defined pointwise by
\begin{equation*}
S_{\varepsilon}(X):=X+\varepsilon\sum_{i=1}^{N}ie^{N}_{i},    
\end{equation*}
where $\varepsilon>0$ denotes the radii of the spheres. As a result, we finally have that
\begin{equation*}
\mathcal{L}=\bigcup_{a\in\mathbb{R}^{3}}\bigcup_{b\in\mathbb{R}^{3}}M(a)\left(\bigcup_{\pi\in\mathfrak{s}(N)}(\pi\circ S_{\varepsilon})\mathcal{Q}\right)+M(b)\mathbf{1}    
\end{equation*}
\begin{figure}
\includegraphics[scale=0.6]{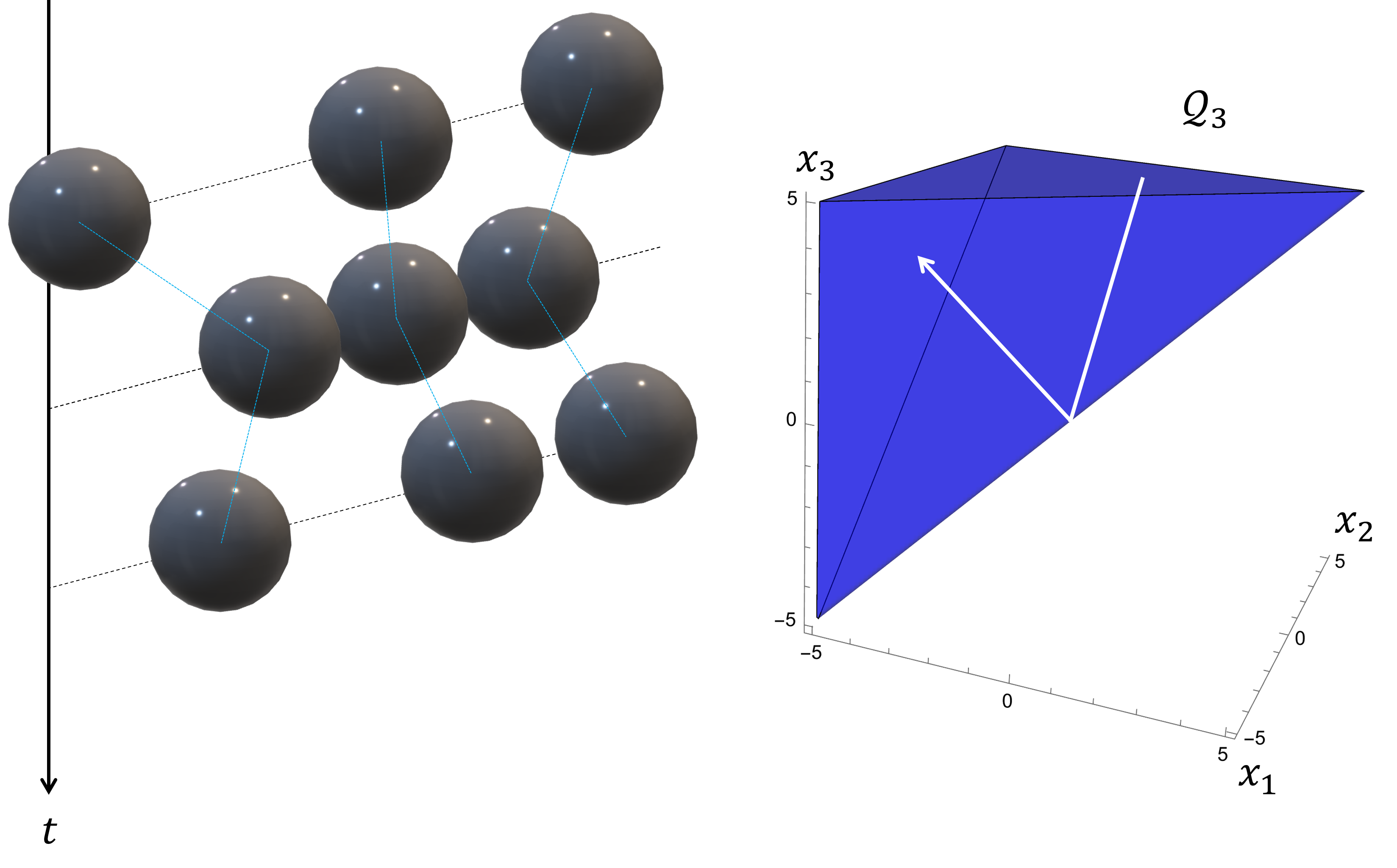}
\caption{The evolution of $N=3$ hard spheres leading to a collision (\emph{left}), with the corresponding billiard trajectory on the fundamental table $\mathcal{Q}$ (\emph{right}, with the fundamental table denoted explicitly by $\mathcal{Q}_{3}$ in this instance that $N=3$).}
\end{figure}The reader can check, for instance, that when $N=3$, the fundamental billiard table $\mathcal{Q}$ is an unbounded wedge generated by two planes with interior angle $\frac{\pi}{3}$ in $\mathbb{R}^{3}$, while in the case $N=4$ it admits the structure of a bundle over a line, whose fibres are unbounded triangular prisms. We restrict our attention to the fundamental billiard table $\mathcal{Q}$ in the sequel, as there are certain simplifications in our analysis which arise from factoring out the radii of the spheres which we model. Adopting the terminology of Joyce \cite{joyce2009manifolds}, we note the following which we state without proof.
\begin{prop}
For any $N\geq 3$, the set $\mathcal{Q}$ admits the structure of a manifold with corners.    
\end{prop}
As we shall introduce trajectories and associated flow operators for billiard dynamics in Section \ref{BAO} below, we also define the tangent bundle $T\mathcal{Q}$ to be
\begin{equation*}
T\mathcal{Q}:=\bigsqcup_{X\in\mathcal{Q}}T_{X}\mathcal{Q},    
\end{equation*}
where the fibres $T_{X}\mathcal{Q}\subseteq\mathbb{R}^{N}$ are 
\begin{equation*}
T_{X}\mathcal{Q}:=\left\{
\begin{array}{ll}
\mathbb{R}^{N} & \quad \text{if}\hspace{2mm}X\in\mathcal{Q}^{\circ}, \vspace{2mm}\\
\mathcal{V}_{X}^{-} & \quad\text{if}\hspace{2mm}X\in\partial\mathcal{Q},
\end{array}
\right.
\end{equation*}
where
\begin{equation*}
\mathcal{V}_{X}^{-}:=\bigcap_{(i, j)\in\gamma_{X}}\left\{V=(v_{1}, ..., v_{N})\in\mathbb{R}^{N}\,:\, v_{i}\geq v_{j}\right\}    
\end{equation*}
and the index set $\gamma_{X}\subset\{1, ..., N\}^{2}$ is defined by 
\begin{equation*}
\gamma_{X}:=\left\{(i, j)\in\{1, ..., N\}^{2}\,:\,x_{i}=x_{j} \right\}.    
\end{equation*}
\subsection{The Universal Invariant Manifold $\mathcal{M}$}
As discussed above, the motivation for this study stems from the fact that any billiard trajectory on $T\mathcal{B}$ which conserves both momentum and energy is not uniquely defined beyond collision for those initial data that lead to it reaching a corner of $\mathcal{B}$. As such, it is natural to distinguish the set of all initial data in $T\mathcal{Q}$ which lead to a simultaneous collision of all $N$ hard spheres.
\begin{defn}
For any $N\geq 3$, we define the set $\mathcal{M}\subset T\mathcal{Q}$ to be
\begin{equation*}
\mathcal{M}=\left\{
Z=(X, V)\in T\mathcal{Q}\,:\,X+tV=c\mathbf{1}\hspace{2mm}\text{for some}\hspace{1mm}t, c\in\mathbb{R}
\right\},
\end{equation*}
where $\mathbf{1}\in\mathbb{R}^{N}$ denotes the vector all of whose entries are 1. 
\end{defn}
As we study the case when all $N$ spheres in the system collide simultaneously, as opposed to $M<N$ spheres colliding simultaneously, the codimension of $\mathcal{M}$ in $T\mathcal{Q}$ is in this sense \emph{maximal}. Like the manifold $T\mathcal{Q}$ in which it is contained, the set $\mathcal{M}$ does not admit the structure of a smooth manifold. Rather, it can be shown that $\mathcal{M}$ admits the structure of a manifold with corners of dimension $N+2$. Its boundary $\partial\mathcal{M}\subset\mathcal{M}$ is given by
\begin{equation*}
\partial\mathcal{M}=\left\{
Z=(X, V)\in\mathcal{M}\,:\,x_{i}=x_{j}\hspace{2mm}\text{and}\hspace{2mm}v_{i}=v_{j}\hspace{2mm}\text{for some}\hspace{2mm}i<j
\right\}.
\end{equation*}
We note that the manifold $\mathcal{M}$ is \emph{universal} as an invariant submanifold of $T\mathcal{Q}$, in the sense that any billiard flow $\{T^{t}\}_{t\in\mathbb{R}}$ on $T\mathcal{Q}$ admits the property that $T^{t}(Z)\in\mathcal{M}$ for all $t\in\mathbb{R}$ if $Z\in\mathcal{M}$. Stated more plainly (and utilising the terminology of Section \ref{BAO} below), no matter the scattering a flow $\{T^{t}\}_{t\in\mathbb{R}}$ on $\mathcal{M}$ admits, the trajectory $t\mapsto T^{t}(Z)$ is a piecewise linear-in-time map admitting a single collision time. 
\subsection{A Natural Measure on $\mathcal{M}$}
It has also been discussed above that one analytical means that can be used to determine a `natural' way to extend billiard trajectories is through studying the properties of measures evolved by the associated flow operators to which the trajectories give rise. One `natural' measure on $\mathcal{M}\subset\mathbb{R}^{2N}$ is the restriction of the $(N+2)$-dimensional Hausdorff measure to $\mathcal{M}$, which we subsequently denote by $\mathscr{H}$. We refer the reader to (\cite{evans2018measure}, Chapter 2) for properties of the Hausdorff measure. There is a convenient formula which allows one to calculate the volume of subsets of the embedded submanifold $\mathcal{M}\subset\mathbb{R}^{2N}$ as measured by $\mathscr{H}$. We recall that if $\Psi:\Omega\rightarrow\mathcal{M}^{\circ}$ is a smooth diffeomorphism for some open $\Omega\subset\mathbb{R}^{N+2}$, then for any bounded open subset $E\subset\mathcal{M}^{\circ}$ it holds that
\begin{equation*}
\mathscr{H}(E)=\int_{\Psi^{-1}(E)}\omega\,d(\mathscr{L}\mres\Psi^{-1}(E)),    
\end{equation*}
where $\omega$ is the \emph{density} associated to the chart map $\Psi$ defined pointwise by
\begin{equation*}
\omega(\zeta):=\sqrt{\mathrm{det}(D\Psi(\zeta)^{T}D\Psi(\zeta))}    
\end{equation*}
for $\zeta\in\Omega$. We shall make frequent use of this representation formula in the sequel. Whilst $\mathscr{H}$ cannot be an invariant measure for billiard flows defined on $\mathcal{M}$ (a claim we prove in Section \ref{derifunky} below), we nevertheless find a canonical measure $\lambda$ on $\mathcal{M}$ which is absolutely continuous with respect to $\mathscr{H}$.
\subsection{Some Existing Results on the `One-dimensional' Model}
We do not aim here to provide a comprehensive review of the literature on billiard dynamics on $T\mathcal{Q}$ or $T\mathcal{P}$, drawing upon only some of the most relevant works to which our main results are pertinent. The existence of momentum- and energy-conserving billiard trajectories on $T\mathcal{B}$ for $N\geq 3$ has been covered, ostensibly, by Alexander in his 1975 Berkeley PhD thesis \cite{alexander1975infinite}. However, his method of construction of global-in-time billiard trajectories on $T\mathcal{B}$ (and therefore on $T\mathcal{Q}$) that conserve both momentum and energy holds only for initial points in $T\mathcal{B}$ outside a measure-theoretically null set, including those initial points that lead to the simultaneous collision of 3 or more hard spheres (\cite{alexander1975infinite}, page 14).  As such, the systematic approach in this article covers the case of initial points his theory cannot handle, but only in the case where motion of the $N$ spheres is directed along a line. Let us also mention that the case of momentum- and energy-conserving dynamics of hard spheres on a line whose radii and masses may differ has also been considered by other authors: see, for instance, the work of Murphy \cite{murphy1994dynamics}. The system of \emph{infinitely-many} hard spheres on a line has also attracted interest from a number of authors, in particular the works of Dobrushin and Fritz \cite{dobrushin1977non}, Lanford \cite{lanford1968classical}, as well as Sinai \cite{sinai1972construction}.

Aside from issues pertaining to the construction of dynamics on $T\mathcal{Q}$, the study of statistical properties of one-dimensional systems has been of interest as a `toy' in the hydrodynamics of particle systems. The study of hydrodynamic limits on $T\mathcal{Q}$ is more tractable than its counterpart study on the full hard sphere phase space $T\mathcal{B}$, and has produced Euler-like equations modelling the large-scale evolution of dynamics on $T\mathcal{Q}$ as $N\rightarrow\infty$ (\cite{kupershmidt1988kinetic}, \cite{boldrighini1983one}, \cite{boldrighini1984hydrodynamics}, \cite{kasperkovitz1985finite}). As this article includes the case $N=3$, let us also mention that ternary collisions in hard sphere systems have also been of recent interest: see, for instance, the work of Ampatzoglou and Pavlovi\'{c} \cite{ampatzoglou2021rigorous} on the derivation of a Boltzmann equation corresponding to dynamics for ternary particle interactions.
\subsection{Notation}\label{notation}
Let us now lay out some notation we employ throughout this article.
\subsubsection{Linear Algebra Conventions}
Suppose $M\geq 1$ and $N\geq 1$ are integers. If $X=(x_{1}, ..., x_{M})\in\mathbb{R}^{M}$ and $V=(v_{1}, ..., v_{M})\in\mathbb{R}^{M}$ are given, we denote their concatenation with round brackets by $(X, V):=(x_{1}, ..., x_{M}, v_{1}, ..., v_{M})\in\mathbb{R}^{2M}$. For any $i\in\{1, ..., M\}$, $e^{M}_{i}\in\mathbb{R}^{M}$ denotes the canonical Euclidean basis vector whose $i\textsuperscript{th}$ component is 1. In turn, $\mathbf{1}:=\sum_{i=1}^{M}e^{M}_{i}$ denotes the vector all of whose entries are 1. If $a\in\mathbb{R}^{M}$ and $b\in\mathbb{R}^{N}$ are given vectors, we write $a\otimes b\in\mathbb{R}^{M\times N}$ to denote the tensor product defined componentwise by $(a\otimes b)_{i, j}:=a_{i}b_{j}$. For $i, j\in\{1, ..., M\}$, $\delta_{i, j}$ denotes the Kronecker delta, while $\widetilde{\delta}_{ij}:=1-\delta_{i, j}$. We denote the identity matrix by $I_{M}\in\mathbb{R}^{M\times M}$ and the zero matrix by $0_{M}\in\mathbb{R}^{M\times M}$. When $M$ and $N$ are different integers, we denote by $0_{M\times N}\in\mathbb{R}^{N\times M}$ the matrix all of whose entries are 0. We denote the general linear group and the orthogonal groups by $\mathrm{GL}(M)$ and $\mathrm{O}(M)$, respectively. If $\{A_{i}\}_{i=1}^{m}\subset\mathbb{R}^{M\times M}$ is a finite sequence of square matrices, then we adopt the convention that
\begin{equation*}
\prod_{i=1}^{m}A_{i}:=A_{m}...A_{1}.    
\end{equation*} 
Finally, if $\Psi:\mathbb{R}^{M}\rightarrow\mathbb{R}^{N}$ is a differentiable map, then $D\Psi:\mathbb{R}^{M}\rightarrow\mathbb{R}^{N\times M}$ denotes the map defined componentwise by $(D\Psi)_{i, j}:=\partial_{j}\Psi_{i}$.
\subsubsection{Function Spaces}
We write $\mathrm{LSC}(\mathbb{R}, \mathbb{R}^{M})$ to denote the set of vector-valued maps whose component functions are lower semi-continuous on $\mathbb{R}$. We write $\mathrm{BV}_{\mathrm{loc}}(\mathbb{R}, \mathbb{R}^{M})$ to denote the set of all maps of locally bounded variation on $\mathbb{R}$: see Ambrosio, Fusco and Pallara \cite{ambrosio2000functions} for further details. In all that follows, we often abuse notation for the argument of functions. For instance, if $F:\mathbb{R}^{2N}\rightarrow \mathbb{R}$ is a function of the concatenated variable $Z=(X, V)$, we shall employ the notation $F(Z)$ and $F(X, V)$ synonymously. 
\subsubsection{Manifolds}
Let $\{A_{x}\,:\,x\in X\}$ be a family of sets with $A_{x}\subseteq Y$ for all $x\in X$. We write $\sqcup_{x\in X}A_{x}$ to denote the disjoint union given by
\begin{equation*}
\bigsqcup_{x\in X}A_{x}:=\left\{(x, a)\,:\,a\in A_{x}\right\}.    
\end{equation*}
We define the canonical projection maps $\Pi_{1}:\sqcup_{x\in X}A_{x}\rightarrow X$ and $\Pi_{2}:\sqcup_{x\in X}A_{x}\rightarrow Y$ pointwise by $\Pi_{1}((x, a)):=x$ and $\Pi_{2}((x, a)):= a$ for all $(x, a)\in\sqcup_{x\in X}A_{x}$.
\subsubsection{Measure Theory and Topology}
We write $\mres$ to denote the operation of measure restriction, and $\#$ to denote the pushforward operator. Given a subset $\Omega\subseteq\mathbb{R}^{M}$, we write $\mathsf{M}(\Omega)$ to denote the set of all Borel measures on $\Omega$. We write $C^{0}_{c}(\mathbb{R}^{M}, \mathbb{R})$ to denote the set of all compactly-supported continuous functions on $\mathbb{R}^{M}$.
\subsection{Rigorous Statements of Main Results}
Before we state our main results, we set out a few basic definitions pertaining to billiard dynamics on $T\mathcal{Q}$. For each integer $N\geq 3$, we write $\mathcal{V}^{-}\subset\mathbb{R}^{N}$ to denote the set of pre-collisional velocities
\begin{equation*}
\mathcal{V}^{-}=\bigcap_{i=1}^{N-1}\left\{V\in\mathbb{R}^{N}\,:\,V\cdot(e^{N}_{i}-e^{N}_{i+1})\geq 0\right\},   
\end{equation*}
and we write $\mathcal{V}^{+}\subset\mathbb{R}^{N}$ to denote the set of post-collisional velocities
\begin{equation*}
\mathcal{V}^{+}=\bigcap_{i=1}^{N-1}\left\{V\in\mathbb{R}^{N}\,:\,V\cdot(e^{N}_{i}-e^{N}_{i+1})\leq 0\right\}.   
\end{equation*}
In what follows, we call any bijection $\sigma:\mathcal{V}^{-}\rightarrow\mathcal{V}^{+}$ an \textbf{$N$-body scattering map}. Intuitively (and as we shall see in Section \ref{BAO} below), an $N$-body scattering map is a map that assigns a post-collisional velocity vector to a pre-collisional velocity vector and is, in turn, used to construct trajectories of an associated family of flow operators. Moreover, we say that $\sigma$ is a \textbf{momentum- and energy-conserving} $N$-body scattering map if and only if it satisfies additionally the conservation of linear momentum
\begin{equation*}
\mathsf{P}(\sigma(V))=\mathsf{P}(V)    
\end{equation*}
together with the conservation of kinetic energy
\begin{equation*}
\mathsf{E}(\sigma(V))=\mathsf{E}(V)    
\end{equation*}
for all $V\in\mathcal{V}^{-}$, where $\mathsf{P}:\mathbb{R}^{N}\rightarrow\mathbb{R}$ is the linear momentum map
\begin{equation}\label{linmommap}
\mathsf{P}(V):=\sum_{i=1}^{N}v_{i}    
\end{equation}
and $\mathsf{E}:\mathbb{R}^{N}\rightarrow\mathbb{R}$ is the energy map
\begin{equation}\label{enmap}
\mathsf{E}(V):=|V|^{2}.    
\end{equation}
We note that there is no need to consider angular momentum in the special case of motion on a line, owing to the fact that the polynomials which characterise the conservation of angular momentum reduce to those of the conservation of linear momentum in this case.

As we shall discover below, the restriction of the $(N+2)$-dimensional Hausdorff measure to $\mathcal{M}$ is \emph{not} an invariant measure for any billiard flow on $\mathcal{M}$. Nevertheless, there is a canonical invariant measure supported on the interior of $\mathcal{M}$ which our calculations reveal in the sequel. As such, we make the following definition for what we term the \emph{Liouville} measure on $\mathcal{M}$.
\begin{defn}[Liouville Measure on $\mathcal{M}$]
For $N\geq 3$, we define the \textbf{Liouville measure on $\mathcal{M}$} $\lambda$ by 
\begin{equation*}
\lambda:=L\mathscr{H},    
\end{equation*}
where the density $L:\mathcal{M}\rightarrow [0, \infty]$ is defined pointwise by 
\begin{equation*}
L(Z):=
\left\{
\begin{array}{ll}
\displaystyle \prod_{i<j}\left((v_{i}-v_{j})^{2}+(x_{i}-x_{j})^{2}\right)^{-\frac{N-2}{N(N-1)}} & \quad \text{if}\hspace{2mm}Z\in\mathcal{M}^{\circ}, \vspace{2mm}\\
\displaystyle \infty & \quad \text{otherwise}.
\end{array}
\right.
\end{equation*}
\end{defn}
\begin{rem}
The reason we have adopted the terminology ``Liouville measure on $\mathcal{M}$'' for $\lambda$ relates to the Liouville Theorem of classical statistical mechanics (see Khinchin \cite{aleksandr1949mathematical}, for example). Indeed, as a consequence of Theorem \ref{firstmainresult} below, our billiard flows $\{T^{t}_{\sigma}\}_{t\in\mathbb{R}}$ admit the property that, for any measurable subset $E\subset\mathcal{M}$, 
\begin{equation*}
\lambda(T^{-t}_{\sigma}(E))=\lambda(E)    
\end{equation*}
for all $t\in\mathbb{R}$. As such, the billiard flows $\{T^{t}_{\sigma}\}_{t\in\mathbb{R}}$ on $\mathcal{M}$ `preserve volumes' as measured by $\lambda$.
\end{rem}
Our first main result characterises all those $N$-body scattering maps which generate a billiard flow on $\mathcal{M}$ that admits the Liouville measure $\lambda$ as an invariant measure.
\begin{thm}\label{firstmainresult}
Let $N\geq 3$. Suppose $\sigma\in C^{0}(\mathcal{V}^{-}, \mathcal{V}^{+})\cap C^{1}((\mathcal{V}^{-})^{\circ}, (\mathcal{V}^{+})^{\circ})$ is a given $N$-body scattering map. It holds that $\sigma$ generates a billiard flow $\{T^{t}_{\sigma}\}_{t\in\mathbb{R}}$ on $\mathcal{M}$ with the property 
\begin{equation*}
T^{t}_{\sigma}\#\lambda=\lambda    
\end{equation*}
for all $t\in\mathbb{R}$ if and only if $\sigma$ is a classical solution of one the second boundary-value problems:
\begin{equation}\label{2ndbvpplus}
\left\{
\begin{array}{l}
\displaystyle \mathrm{det}(D\sigma(V))=\frac{H(V)}{H(\sigma(V))} \quad \text{for}\hspace{2mm}V\in(\mathcal{V}^{-})^{\circ}, \vspace{2mm}\\
\displaystyle \sigma(\mathcal{V}^{-})=\mathcal{V}^{+}
\end{array}
\right.
\end{equation}
or
\begin{equation}\label{2ndbvpminus}
\left\{
\begin{array}{l}
\displaystyle \mathrm{det}(D\sigma(V))=-\frac{H(V)}{H(\sigma(V))} \quad \text{for}\hspace{2mm}V\in(\mathcal{V}^{-})^{\circ}, \vspace{2mm}\\
\displaystyle \sigma(\mathcal{V}^{-})=\mathcal{V}^{+},
\end{array}
\right.
\end{equation}
where $H$ is defined pointwise by
\begin{equation}\label{RHS}
H(W):=\frac{\left(\displaystyle \sum_{i<j}(w_{i}-w_{j})^{2}\right)^{1/2}}{\left(\displaystyle \prod_{i<j}(w_{i}-w_{j})^{2}\right)^{\frac{N-2}{N(N-1)}}} 
\end{equation}
for $W=(w_{1}, ..., w_{N})\in(\mathcal{V}^{-})^{\circ}\cup(\mathcal{V}^{+})^{\circ}$.
\end{thm}
The Jacobian PDE in the second boundary-value problems \eqref{2ndbvpplus} and \eqref{2ndbvpminus} are highly nonlinear and therefore challenging to solve. We shall identify \emph{linear} solutions $\sigma$ of these problems below. 

Our second main result finds sufficient conditions on the form of the density of an invariant measure of a billiard flow on $\mathcal{M}$ which is absolutely continuous with respect to the Liouville measure. 
\begin{thm}\label{densitythm}
Let $N\geq 3$. Suppose $\sigma\in C^{0}(\mathcal{V}^{-}, \mathcal{V}^{+})\cap C^{1}((\mathcal{V}^{-})^{\circ}, (\mathcal{V}^{+})^{\circ})$ is a momentum- and energy-conserving $N$-body scattering map, and let $\{T^{t}_{\sigma}\}_{t\in\mathbb{R}}$ denote the billiard flow on $\mathcal{M}$ it generates. Suppose, moreover, that $\mu\ll \lambda$, with
\begin{equation*}
\frac{d\mu}{d\lambda}=M    
\end{equation*}
for some measurable function $M:\mathcal{M}\rightarrow\mathbb{R}$. If
$M$ is of the form
\begin{equation*}
M(Z)=m_{0}(X\cdot(\mathsf{E}(V)\mathbf{1}-\mathsf{P}(V)V), V)    
\end{equation*}
for $Z=(X, V)\in\mathcal{M}$ and some measurable $m_{0}:\mathbb{R}\times\mathbb{R}^{N}\rightarrow\mathbb{R}$ admitting the scattering symmetry
\begin{equation}\label{symcon}
m_{0}(y, \sigma(V))=m_{0}(y, V)    
\end{equation}
for $y\in\mathbb{R}$ and $V\in\mathcal{V}^{-}$, then $T^{t}_{\sigma}\#\mu=\mu$ for all $t\in\mathbb{R}$.
\end{thm}
Finally, our third main result identifies the unique linear $N$-body scattering map that conserves momentum and energy, and in turn shows that the billiard flow it generates admits the Liouville measure on $\mathcal{M}$ as an invariant measure.
\begin{thm}\label{uniquelinear}
Let $N\geq 3$. There exists a unique linear momentum- and energy-conserving $N$-body scattering map $\sigma^{\star}:\mathcal{V}^{-}\rightarrow\mathcal{V}^{+}$ given by
\begin{equation}\label{linscatmap}
\sigma^{\star}(V):=\left(\frac{2}{N}\mathbf{1}\otimes\mathbf{1}-I_{N}\right)V    
\end{equation}
for all $V\in\mathcal{V}^{-}$. Moreover, $\sigma^{\star}$ generates a momentum- and energy-conserving billiard dynamics $\{T^{t}_{\star}\}_{t\in\mathbb{R}}$ on $\mathcal{M}$ with the property that
\begin{equation*}
T^{t}_{\star}\#\lambda=\lambda    
\end{equation*}
for all $t\in\mathbb{R}$.
\end{thm}
\begin{rem}
We note that, even though the case of binary interactions is not considered in this article, the formula \eqref{linscatmap} recovers the case of the unique linear 2-body scattering map
\begin{equation*}
\sigma^{\star}(V)=\left(
\begin{array}{cc}
0 & 1 \\
1 & 0
\end{array}
\right)V    
\end{equation*}
which conserves linear momentum and energy when $N=2$.
\end{rem}
\subsection{Structure of the Paper}
In Section \ref{BAO}, we introduce some of the basic objects with which we work throughout. In Sections \ref{emmteeminus} and \ref{emmplus}, in what amounts to a rather involved application of the Change of Variables theorem, we process the information in those integrals under study not only to derive the PDE characterising $N$-body scattering maps of billiard flows on $\mathcal{M}$ which admit the Liouville measure $\lambda$ on $\mathcal{M}$ as an invariant measure, but also a functional equation whose solutions generate measures which are absolutely-continuous with respect to $\lambda$. In Section \ref{uniquelinny}, we find the unique momentum- and energy-conserving $N$-body scattering which is a linear map and show, in turn, the billiard flow on $\mathcal{M}$ it generates admits $\lambda$ as an invariant measure. 
\section{Basic Auxiliary Objects}\label{BAO}
We introduce the basic objects of study in what follows. For notational convenience, and to avoid decorating symbols for maps and spaces with too many sub- and superscripts, we do not in general denote dependence of objects on the number of spheres $N\geq 3$ under study. 
\subsection{A Chart for the Interior of $\mathcal{M}$}\label{chartinfo}
Moving forward, owing to the fact we shall build integrals over $\mathcal{M}$ with respect to the Hausdorff measure which are insensitive to the values of maps on the boundary of $\mathcal{M}$, we focus our attention on the interior of $\mathcal{M}$, namely $\mathcal{M}^{\circ}=\mathcal{M}\setminus\partial\mathcal{M}$. In particular, we shall make extensive use of the \emph{chart map} $\Psi:\Omega\rightarrow\mathcal{M}^{\circ}$ which charts the interior of the codimension $(N-2)$ subset $\mathcal{M}\subset T\mathcal{Q}$ of $\mathbb{R}^{2N}$ pointwise by
\begin{equation*}
\Psi(\zeta):=\sum_{i=1}^{N}(e_{i}^{2N}\otimes e_{i}^{2N})\zeta+\sum_{i=1}^{N}(e_{i+N}^{2N}\otimes e_{i}^{N})\psi(\zeta),    
\end{equation*}
for $\zeta=(X, U)\in\Omega$, where
\begin{equation*}
\Omega:=\mathcal{Q}^{\circ}\times(\mathbb{R}^{2}\setminus\Delta)    
\end{equation*}
and $X=(x_{1}, ..., x_{N})\in\mathcal{Q}^{\circ}$ with $U=(u_{1}, u_{2})\in\mathbb{R}^{2}\setminus\Delta$, where $\Delta\subset\mathbb{R}^{2}$ denotes the diagonal
\begin{equation*}
\Delta:=\left\{U=(u_{1}, u_{2})\in\mathbb{R}^{2}\,:\,u_{1}=u_{2}\right\},    
\end{equation*}
and where the \emph{velocity map} $\psi:\Omega\rightarrow\mathbb{R}^{N}$ is defined pointwise by
\begin{equation*}
\psi_{i}(\zeta):=
\left\{
\begin{array}{ll}
\displaystyle u_{i} & \quad \text{if}\hspace{2mm}i\in\{1, 2\}, \vspace{2mm}\\
\displaystyle \frac{(x_{i}-x_{2})u_{1}-(x_{i}-x_{1})u_{2}}{x_{1}-x_{2}} & \quad \text{if} \hspace{2mm} i\in\{3, ..., N\}.
\end{array}
\right.
\end{equation*}
Manifestly, it holds that $\Psi\in C^{\infty}(\Omega, \mathcal{M}^{\circ})$ and one can check that it is a smooth diffeomorphism. Using this chart, $\mathcal{M}^{\circ}$ may be characterised in bundle notation as
\begin{equation*}
\mathcal{M}^{\circ}=\bigsqcup_{X\in\mathcal{Q}^{\circ}}F_{X},    
\end{equation*}
where the fibres $F_{X}\subset\mathbb{R}^{N}$ are defined for each $X$ by
\begin{equation*}
F_{X}=\left\{
\psi(\zeta)\in\mathbb{R}^{N}\,:\,U\in\mathbb{R}^{2}
\right\},
\end{equation*}
noting that $\zeta=(X, U)$.
\begin{rem}
There are many `natural' ways by which one can define a chart map $\Psi$ for the interior of $\mathcal{M}$, and our particular choice which slaves the velocity the spheres labelled by $i\in\{3, ..., N\}$ to those labelled by 1 and 2 admits a signature in our calculations, as we shall see in the sequel. 
\end{rem}
\subsection{Billiard Flows on $\mathcal{M}$}
In this Section, we set out the analytical properties we require of billiard flows on the manifold $\mathcal{M}$ in all that follows. We do not concern ourselves with the construction of flows defined on the full phase space $T\mathcal{Q}$ as this has been handled by other authors, viz. Alexander \cite{alexander1975infinite}. The systematic study of dynamics on $\mathcal{M}$ has not, to our knowledge, received attention in the literature.
\subsubsection{Momentum- and Energy-conserving Billiard Flows on $T\mathcal{Q}$}
We begin with the following basic definition of billiard flow.
\begin{defn}[Billiard Flow]
We say that a one-parameter family $\{T^{t}\}_{t\in\mathbb{R}}$ of maps $T^{t}:T\mathcal{Q}\rightarrow T\mathcal{Q}$ is a \textbf{billiard flow} on $T\mathcal{Q}$ if and only if for each $Z\in T\mathcal{Q}$, it holds that
\begin{itemize}
\item the map $t\mapsto (\Pi_{1}\circ T^{t})(Z)$ is a piecewise linear map of class $C^{0}(\mathbb{R}, T\mathcal{Q})$ which satisfies
\begin{equation*}
\frac{d}{dt_{\pm}}(\Pi_{1}\circ T^{t})(Z)=\lim_{s\rightarrow t\pm}(\Pi_{2}\circ T^{s})(Z)    
\end{equation*}
for all $t\in\mathbb{R}$;
\item the map $t\mapsto (\Pi_{2}\circ T^{t})(Z)$ is of class $\mathrm{LSC}(\mathbb{R}, \mathbb{R}^{N})\cap \mathrm{BV}_{\mathrm{loc}}(\mathbb{R}, \mathbb{R}^{N})$.
\end{itemize}
We term the members $T^{t}$ of the billiard flow $\{T^{t}\}_{t\in\mathbb{R}}$ \textbf{flow maps}.
\end{defn}
We shall work exclusively on the manifold $\mathcal{M}\subset T\mathcal{Q}$. We subsequently abuse notation and denote the restriction $T^{t}|_{\mathcal{M}}$ of any flow map $T^{t}$ to the manifold $\mathcal{M}$ simply by $T^{t}$. As we focus our attention only on those billiard flows which conserve momentum and energy, we also require the following definition.
\begin{defn}[Momentum- and Energy-conserving Billiard Flow]
We say that a billiard flow $\{T^{t}\}_{t\in\mathbb{R}}$ on $T\mathcal{Q}$ is \textbf{momentum- and energy-conserving} if and only if for each $Z\in T\mathcal{Q}$ it holds that
\begin{equation*}
\mathsf{P}((\Pi_{2}\circ T^{t})(Z))=\mathsf{P}(\Pi_{2}(Z))    
\end{equation*}
and
\begin{equation*}
\mathsf{E}((\Pi_{2}\circ T^{t})(Z))=\mathsf{E}(\Pi_{2}(Z))    
\end{equation*}
for all $t\in\mathbb{R}$.
\end{defn}
Let us once again mention that, in the case of hard sphere motion on a line, the conservation of angular momentum is redundant if the conservation of linear momentum is enforced.
\subsubsection{Scattering Maps for Billiard Flows on $\mathcal{M}$}
We now set out some definitions which are relevant to the simultaneous collision of $N$ hard spheres. 
\begin{defn}[Collision Time Map on $\mathcal{M}^{\circ}$]
Let $N\geq 3$. We write $\tau:\mathcal{M}^{\circ}\rightarrow\mathbb{R}$ to denote the \textbf{collision time map on $\mathcal{M}^{\circ}$} pointwise by
\begin{equation}\label{coll1and2}
\tau(Z):=-\frac{x_{2}-x_{1}}{v_{2}-v_{1}}    
\end{equation}
for $Z=(X, V)\in\mathcal{M}^{\circ}$.
\end{defn}
We note that the definition of $\tau$ appears to involve only those variables indexed by 1 and 2. However, it can be checked that
\begin{equation*}
\tau(Z)=-\frac{x_{i}-x_{j}}{v_{i}-v_{j}}    
\end{equation*}
for any $i, j\in\{1, ..., N\}$ with $i<j$. Nevertheless, owing to our explicit choice of chart map for $\mathcal{M}^{\circ}$ given in Section \ref{chartinfo} above, we shall adopt the definition of the collision time map \eqref{coll1and2} involving the indices 1 and 2. With this definition in hand, we may now define what we mean by a \emph{scattering map} of a billiard flow $\{T^{t}\}_{t\in\mathbb{R}}$ on $\mathcal{M}$.
\begin{defn}[Scattering Map of a Billiard Flow on $\mathcal{M}$]
Let $N\geq 3$. Suppose that $\{T^{t}\}_{t\in\mathbb{R}}$ is a billiard flow on $T\mathcal{Q}$. Its associated \textbf{scattering map} $\sigma:\mathcal{V}^{-}\rightarrow\mathcal{V}^{+}$ is defined pointwise by
\begin{equation*}
\sigma(V):=\lim_{t\downarrow\tau(Z)}(\Pi_{2}\circ T^{t})(Z)    
\end{equation*}
for $Z=(0, V)\in\mathcal{M}$.
\end{defn}
While every billiard flow gives rise to a scattering map, the converse also holds true. We require the following result, the details of whose proof we defer to Section \ref{reppy} below.
\begin{prop}\label{sigmagenerate}
Every scattering map $\sigma$ generates a billiard flow on $\mathcal{M}$, i.e. for any scattering map $\sigma:\mathcal{V}^{-}\rightarrow\mathcal{V}^{+}$ there exists a billiard flow $\{T^{t}_{\sigma}\}_{t\in\mathbb{R}}$ on $\mathcal{M}$ with the property
\begin{equation*}
\sigma(V)=\lim_{t\downarrow \tau(Z)}(\Pi_{2}\circ T^{t}_{\sigma})(Z)    
\end{equation*}
for $Z=(0, V)\in\mathcal{M}$.
\end{prop}
\subsubsection{Regular Billiard Flows on $\mathcal{M}$}
Given a billiard flow on $\mathcal{M}$, its scattering map $\sigma$ may possess no degree of regularity at all. As our criterion on a flow admitting the Liouville measure $\lambda$ as an invariant measure features a first-order PDE (c.f. the statement of Theorem \ref{firstmainresult}), we require the following notion of \emph{regularity} for the billiard flows with which we work.
\begin{defn}[Regular Billiard Flow]
We say that a billiard flow $\{T^{t}_{\sigma}\}_{t\in\mathbb{R}}$ on $\mathcal{M}$ is a \textbf{regular} billiard flow on $\mathcal{M}$ if and only if its scattering map $\sigma$ is of class $C^{0}(\mathcal{V}^{-},\mathcal{V}^{+})\cap C^{1}((\mathcal{V}^{-})^{\circ}, (\mathcal{V}^{+})^{\circ})$.
\end{defn}
\subsection{Billiard Trajectories on $\mathcal{M}$}
In this Section, we define billiard trajectories on $\mathcal{M}$ which start from initial data in the interior $\mathcal{M}^{\circ}$. In loose terms, this corresponds to those initial data which model the centres of all $N$ spheres being a strictly positive distance from one another at time zero. In turn, we shall use these explicit trajectory formulae to write down representation formulae for the associated billiard flow maps in Section \ref{reppy} below. For the moment, we must split our considerations into 2 cases, namely one in which initial data $Z_{0}\in\mathcal{M}$ are \emph{pre}-collisional, and one in which the data $Z_{0}$ are \emph{post}-collisional. 

Suppose that $\sigma\in C^{0}(\mathcal{V}^{-}, \mathcal{V}^{+})\cap C^{1}((\mathcal{V}^{-})^{\circ}, (\mathcal{V}^{+})^{\circ})$ is a given $N$-body scattering map. Let $Z_{0}=(X_{0}, V_{0})\in\mathcal{M}^{\circ}$ be such that $V_{0}\in(\mathcal{V}^{-})^{\circ}$. The associated \emph{pre-collisional} billiard trajectory $Z_{\sigma}(\cdot; Z_{0}):\mathbb{R}\rightarrow\mathcal{M}$ is defined to be
\begin{equation}\label{precolltraj}
\begin{array}{c}
\displaystyle Z_{\sigma}(t; Z_{0}):= \vspace{2mm}\\
\displaystyle \left\{
\begin{array}{ll}
\displaystyle \left(I_{2N}+t\sum_{i=1}^{2N}e^{2N}_{i}\otimes e^{2N}_{i+N}\right)Z_{0} & \quad \text{if}\hspace{2mm}-\infty<t\leq \tau(Z_{0}), \vspace{2mm} \\
\displaystyle  \left(I_{2N}+(t-\tau(Z_{0}))\sum_{i=1}^{N}e_{i}^{2N}\otimes e_{i+N}^{2N}\right)\Sigma_{\sigma}(Z_{0})+t\sum_{i=1}^{N}e_{i}^{2N}\otimes e_{i+N}^{2N}Z_{0} & \quad \text{if}\hspace{2mm}\tau(Z_{0})<t<\infty,
\end{array}
\right.
\end{array}
\end{equation}
where $\Sigma_{\sigma}:\mathcal{M}\rightarrow\mathbb{R}^{2N}$ denotes the \emph{extended scattering map} associated to the $N$-body scattering $\sigma$ defined by
\begin{equation*}
\Sigma_{\sigma}(Z_{0}):=\left(
\begin{array}{c}
X_{0} \vspace{1mm}\\
\sigma(V_{0})
\end{array}
\right)
\end{equation*}
for $Z_{0}=(X_{0}, V_{0})\in\mathcal{M}$ such that $V_{0}\in\mathcal{V}^{-}$. Similarly, let $Z_{0}=(X_{0}, V_{0})\in\mathcal{M}^{\circ}$ be such that $V_{0}\in (\mathcal{V}^{+})^{\circ}$. The associated \emph{post-collisional} billiard trajectory $Z_{\sigma}(\cdot; Z_{0}):\mathbb{R}\rightarrow\mathcal{M}$ is defined by
\begin{equation}\label{postcolltraj}
\begin{array}{c}
\displaystyle Z_{\sigma}(t; Z_{0}):= \vspace{2mm}\\
\displaystyle \left\{
\begin{array}{ll}
\displaystyle \left(I_{2N}+(t-\tau(Z_{0}))\sum_{i=1}^{N}e_{i}^{2N}\otimes e_{i+N}^{2N}\right)\Sigma_{\sigma^{-1}}(Z_{0})+t\sum_{i=1}^{N}e_{i}^{2N}\otimes e_{i+N}^{2N}Z_{0} & \quad \text{if}\hspace{2mm}-\infty<t\leq \tau(Z_{0}), \vspace{2mm} \\
\displaystyle \left(I_{2N}+t\sum_{i=1}^{2N}e^{2N}_{i}\otimes e^{2N}_{i+N}\right)Z_{0} & \quad \text{if}\hspace{2mm}\tau(Z_{0})<t<\infty,
\end{array}
\right.
\end{array}
\end{equation}
where $\Sigma_{\sigma^{-1}}:\mathcal{M}\rightarrow\mathbb{R}^{2N}$ denotes the extended scattering map associated to the inverse $\sigma^{-1}$, given pointwise by
\begin{equation*}
\Sigma_{\sigma}(Z_{0}):=\left(
\begin{array}{c}
X_{0} \vspace{1mm}\\
\sigma^{-1}(V_{0})
\end{array}
\right)    
\end{equation*}
for $Z_{0}=(X_{0}, V_{0})\in\mathcal{M}$ such that $V_{0}\in\mathcal{V}^{+}$. We would like to find a billiard flow $\{T^{t}\}_{t\in\mathbb{R}}$ on $\mathcal{M}$ whose associated flow maps $T^{t}:\mathcal{M}\rightarrow\mathcal{M}$ admit the property that
\begin{equation*}
T^{t}(Z_{0})=Z_{\sigma}(t; Z_{0})    
\end{equation*}
for all $Z_{0}\in\mathcal{M}^{\circ}$ and each \emph{fixed} $t\in\mathbb{R}$. The reader will note that the trajectory formulae \eqref{precolltraj} and \eqref{postcolltraj} do not immediately suggest the form of the flow maps $T^{t}_{\sigma}$ to which they are related. We shall consider this below.
\subsection{Representation Formulae for Flow Maps on $\mathcal{M}^{\circ}$}\label{reppy}
In this Section, we offer a sketch of the details involved in a proof of Proposition \ref{sigmagenerate} above, namely that every $N$-body scattering map $\sigma$ generates an associated billiard flow $\{T^{t}_{\sigma}\}_{t\in\mathbb{R}}$ on $\mathcal{M}$. However, as it is all we require, we only consider here the construction of the billiard flow maps on the interior of $\mathcal{M}$ in what follows. 

Owing to the role that the inverse $\sigma^{-1}$ of a given $N$-body scattering map $\sigma$ plays in definition of billiard trajectories, moving forward we must distinguish between the cases $t>0$ and $t<0$, with the case $t=0$ being trivial in that $T^{0}_{\sigma}:=\mathrm{id}_{\mathcal{M}}$ for any choice of $N$-body scattering map $\sigma$. We begin by defining the auxiliary subset $\mathcal{U}_{t}^{-}\subset\Omega$ by
\begin{equation*}
\mathcal{U}_{t}^{-}:=\left\{
\begin{array}{ll}
\displaystyle \bigsqcup_{X\in\mathcal{Q}^{\circ}}\left\{U\in\mathbb{R}^{2}\setminus\Delta\,:\,u_{1}\leq u_{2}+\frac{x_{2}-x_{1}}{t}\right\} & \quad \text{if}\hspace{2mm}t>0, \vspace{2mm}\\
\mathcal{Q}^{\circ}\times\left\{U\in\mathbb{R}^{2}\,:\,u_{1}< u_{2}\right\} & \quad \text{if}\hspace{2mm}t=0, \vspace{2mm}\\
\displaystyle \bigsqcup_{X\in\mathcal{Q}^{\circ}}\left\{
U\in\mathbb{R}^{2}\setminus\Delta\,:\,u_{1}> u_{2}+\frac{x_{2}-x_{1}}{t}\right\} & \quad \text{if}\hspace{2mm}t<0,
\end{array}
\right.
\end{equation*}
and its counterpart $\mathcal{U}_{t}^{+}\subset\Omega$ by 
\begin{equation*}
\mathcal{U}_{t}^{+}:=\left\{
\begin{array}{ll}
\displaystyle \bigsqcup_{X\in\mathcal{Q}^{\circ}}\left\{U\in\mathbb{R}^{2}\setminus\Delta\,:\,u_{1}> u_{2}+\frac{x_{2}-x_{1}}{t}\right\} & \quad \text{if}\hspace{2mm}t>0, \vspace{2mm}\\
\mathcal{Q}^{\circ}\times\left\{U\in\mathbb{R}^{2}\,:\,u_{1}>u_{2}\right\} & \quad \text{if}\hspace{2mm}t=0, \vspace{2mm}\\
\displaystyle \bigsqcup_{X\in\mathcal{Q}^{\circ}}\left\{
U\in\mathbb{R}^{2}\setminus\Delta\,:\,u_{1}\leq u_{2}+\frac{x_{2}-x_{1}}{t}\right\} & \quad \text{if}\hspace{2mm}t<0.
\end{array}
\right.
\end{equation*}
We note that $\mathcal{U}_{t}^{-}\cup\mathcal{U}_{t}^{+}=\Omega$. Using the subsets $\mathcal{U}_{t}^{-}$ and $\mathcal{U}_{t}^{+}$ of the chart map domain $\Omega=\mathcal{Q}^{\circ}\times(\mathbb{R}^{2}\setminus\Delta)$, we now define the important sets $\mathcal{M}_{t}^{-}$ and $\mathcal{M}_{t}^{+}$ which will be of focus in Sections \ref{emmteeminus} and \ref{emmplus}, respectively.
\begin{defn}
We define $\mathcal{M}_{t}^{-}\subset\mathcal{M}$ and $\mathcal{M}_{t}^{+}\subset\mathcal{M}$ to be
\begin{equation*}
\mathcal{M}_{t}^{\pm}:=\Psi(\mathcal{U}_{t}^{\pm})    
\end{equation*}
for $t\in\mathbb{R}$.
\end{defn}
In particular, it holds that $\Psi(\mathcal{U}_{t}^{-}\cup\mathcal{U}_{t}^{+})=\mathcal{M}^{\circ}$ for all $t\in\mathbb{R}$. With these definitions in hand, we may now posit Ans\"{a}tze for the form of the billiard flow operators $T^{t}_{\sigma}$ associated to the trajectories \eqref{precolltraj} and \eqref{postcolltraj} above. The reader may verify that the maps $T^{t}_{\sigma}:\mathcal{M}^{\circ}\rightarrow\mathcal{M}$ defined by
\begin{equation*}
\begin{array}{c}
\displaystyle T^{t}_{\sigma}(Z_{0}):=\vspace{2mm}\\
\displaystyle \left\{
\begin{array}{ll}
\displaystyle \left(
I_{2N}+t\sum_{i=1}^{2N}e^{2N}_{i}\otimes e^{2N}_{i+N}
\right)Z_{0} & \quad \text{for}\hspace{2mm}Z_{0}\in\mathcal{M}_{t}^{-},\vspace{2mm}\\
\displaystyle  \left(I_{2N}+(t-\tau(Z_{0}))\sum_{i=1}^{N}e_{i}^{2N}\otimes e_{i+N}^{2N}\right)\Sigma_{\sigma}(Z_{0})+t\sum_{i=1}^{N}e_{i}^{2N}\otimes e_{i+N}^{2N}Z_{0} & \quad \text{for}\hspace{2mm}Z_{0}\in\mathcal{M}_{t}^{+} 
\end{array}
\right.
\end{array}
\end{equation*}
for fixed $t>0$, and by
\begin{equation*}
\begin{array}{c}
\displaystyle T^{t}_{\sigma}(Z_{0}):=\vspace{2mm}\\
\displaystyle \left\{
\begin{array}{ll}
\displaystyle \left(
I_{2N}+t\sum_{i=1}^{2N}e^{2N}_{i}\otimes e^{2N}_{i+N}
\right)Z_{0} & \quad \text{for}\hspace{2mm}Z_{0}\in\mathcal{M}_{t}^{-}, \vspace{2mm}\\
\displaystyle \left(I_{2N}+(t-\tau(Z_{0}))\sum_{i=1}^{N}e_{i}^{2N}\otimes e_{i+N}^{2N}\right)\Sigma_{\sigma^{-1}}(Z_{0})+t\sum_{i=1}^{N}e_{i}^{2N}\otimes e_{i+N}^{2N}Z_{0} & \quad \text{for}\hspace{2mm}Z_{0}\in\mathcal{M}_{t}^{+}
\end{array}
\right.
\end{array}
\end{equation*}
for fixed $t<0$ admit the property that
\begin{equation*}
T^{t}_{\sigma}(Z_{0})=Z_{\sigma}(t; Z_{0})    
\end{equation*}
for any $Z_{0}\in\mathcal{M}^{\circ}$ and all $t\in\mathbb{R}$. In particular, they admit the desired property that
\begin{equation*}
\sigma(V_{0})=\lim_{t\downarrow\tau(Z_{0})}(\Pi_{2}\circ T^{t}_{\sigma})(Z_{0})    
\end{equation*}
for $Z_{0}=(0, V_{0})$ with $V_{0}\in(\mathcal{V}^{-})^{\circ}$, as well as
\begin{equation*}
\sigma^{-1}(V_{0})=\lim_{t\uparrow\tau(Z_{0})}(\Pi_{2}\circ T^{t}_{\sigma})(Z_{0})  
\end{equation*}
for $Z_{0}=(0, V_{0})$ with $V_{0}\in(\mathcal{V}^{+})^{\circ}$. As the precise properties of the $N$-body scattering $\sigma$ have not yet been prescribed, other than the prescription that $\sigma$ map the polytope $\mathcal{V}^{-}$ to the polytope $\mathcal{V}^{+}$, we cannot say anything about the image sets $T^{t}_{\sigma}(\mathcal{M}_{t}^{+})$ for $t\in\mathbb{R}\setminus\{0\}$. As such, we make the following assumptions on the properties of any $N$-body scattering map throughout this article.
\begin{ass}[Admissible $N$-body Scattering Map $\sigma$]\label{Ass}
We say that an $N$-body scattering map $\sigma:\mathcal{V}^{-}\rightarrow\mathcal{V}^{+}$ is \textbf{admissible} if and only if it satisfies the following two conditions:
\begin{itemize}
    \item $\sigma^{-1}(V)=-\sigma(-V)$ for all $V\in\mathcal{V}^{+}$;
    \item $T^{t}_{\sigma}(\mathcal{M}^{-}_{t})\cup T^{t}_{\sigma}(\mathcal{M}^{+}_{t})=\mathcal{M}^{\circ}$ for all $t\in\mathbb{R}$.
\end{itemize}
\end{ass}
The first of the above assumptions, which expresses the inverse scattering map $\sigma$ in terms of the scattering map itself, shall be employed in the proof of uniqueness of linear $N$-body scattering that conserves both linear momentum and energy in Section \ref{uniquelinny} below. It is equivalent to the assumption that the billiard flow $\{T^{t}_{\sigma}\}_{t\in\mathbb{R}}$ generated by $\sigma$ admits no \emph{arrow of time}: see Landau and Lifschitz \cite{landau2013statistical}. The second of these assumptions ensures that the billiard flow maps do not map regions of $\mathcal{M}$ of positive measure to those of null measure (which can occur if $\sigma$ does not conserve the energy of its argument, for instance). In the sequel, we shall refer to an \emph{admissible} $N$-body scattering map simply as an $N$-body scattering map.
\begin{rem}
In what follows, we subsequently drop the `0' subscript from initial data in $\mathcal{M}$, thereby considering flow maps $T^{t}_{\sigma}$ as a function of $Z=(X, V)\in\mathcal{M}$ rather than of $Z_{0}=(X_{0}, V_{0})$.    
\end{rem}
\section{Analysis over the Region $\mathcal{M}_{t}^{-}$}\label{emmteeminus}
In this Section, we consider the properties of billiard flow maps $T^{t}_{\sigma}$ only over the region $\mathcal{M}_{t}^{-}$ as defined in Section \ref{reppy} above. We consider properties of flow maps over the complement $\mathcal{M}_{t}^{+}$ in Section \ref{emmplus} below. Our aim here is to find symmetry properties of any measure $\mu\in\mathsf{M}(\mathcal{M})$ with $\mu\ll\mathscr{H}$ which satisfies the identity
\begin{equation}\label{simplebrackets}
\langle T_{\sigma}^{t}\#\mu, \Phi\rangle=\langle\mu, \Phi\rangle    
\end{equation}
for all $\Phi\in C^{0}_{c}(\mathbb{R}^{2N}, \mathbb{R})$ with the property that the support of $\Phi$ satisfies the condition
\begin{equation}\label{suup}
\mathscr{H}(\mathrm{supp}(\Phi)\cap\mathcal{M}_{t}^{-})> 0.    
\end{equation}
Indeed, we invoke the following basic lemma which motivates this aim.
\begin{lem}
Suppose $N\geq 3$. It holds that for any bounded measurable subset $E\subset\mathcal{M}$ there exists a sequence $\{\Phi_{j}\}_{j=1}^{\infty}\subset C^{0}_{c}(\mathbb{R}^{2N}, \mathbb{R})$ with the property that
\begin{equation*}
\mu(E)=\lim_{j\rightarrow\infty}\int_{\mathcal{M}}\Phi_{j}\,d\mu    
\end{equation*}
\end{lem}
\begin{proof}
This is a consequence of Urysohn's Lemma: see Munkres (\cite{munkres2018topology}, Chapter 4).
\end{proof}
We shall show in the sequel that identity \eqref{simplebrackets} does \emph{not} hold if the measure $\mu$ is taken to be the Hausdorff measure $\mathscr{H}$. Rather, we identify a function $L$ derived from the `shape' of the manifold $\mathcal{M}$ such that $\mu=L\mathscr{H}$ satisfies \eqref{simplebrackets}.
\subsection{Integration over Codimension $(N-2)$ Surfaces in $\mathbb{R}^{2N}$}
We state without proof the following elementary result on the evaluation of codimension $(N-2)$ surface integrals in $\mathbb{R}^{2N}$ by way of parametrisation maps.
\begin{prop}
Suppose $N\geq 3$, and that $\mu\ll\mathscr{H}$ is a given measure on $\mathcal{M}$ with density $M$. Suppose further that $\Psi:\Omega\rightarrow\mathcal{M}^{\circ}$ denotes any diffeomorphism of class $C^{1}$. For any $\Phi\in C^{0}_{c}(\mathbb{R}^{2N}, \mathbb{R})$ with the property
\begin{equation*}
\mathscr{H}(\mathrm{supp}(\Phi)\cap\mathcal{M})>0,    
\end{equation*}
it holds that
\begin{equation*}
\int_{\mathcal{M}}\Phi\,d\mu=\int_{\Omega}(\Phi\circ\Psi)(M\circ\Psi)\omega\,d\mathscr{L},    
\end{equation*}
where the surface density $\omega\in C^{0}(\Omega, [0, \infty))$ is given by
\begin{equation*}
\omega(X, U):=\sqrt{D\Psi(X, U)^{T}D\Psi(X, U)}    
\end{equation*}
for $(X, U)\in\Omega$, and where $\mathscr{L}$ denotes the restriction of the Lebesgue measure on $\mathbb{R}^{N+2}$ to the Lebesgue-measurable subset $\Omega\subset\mathbb{R}^{N+2}$.
\end{prop}
\begin{proof}
See Duistermaat and Kolk (\cite{duistermaat2004multidimensional}, Section 7.3) and Evans and Gariepy (\cite{evans2018measure}, Section 3.3.4.D).   
\end{proof}
We recall from above that for any momentum- and energy-conserving $N$-body scattering map $\sigma$, the billiard flow $\{T^{t}_{\sigma}\}_{t\in\mathbb{R}}$ it generates has flow maps of the shape
\begin{equation*}
T^{t}_{\sigma}(Z)=\left\{
\begin{array}{ll}
\displaystyle \left(I_{2N}+t\sum_{i=1}^{2N}e_{i}^{N}\otimes e_{i+N}^{2N}\right)Z & \quad \text{if}\hspace{2mm}Z\in\mathcal{M}_{t}^{-}, \vspace{2mm}\\
\displaystyle \left(I_{2N}+(t-\tau(Z))\sum_{i=1}^{N}e_{i}^{2N}\otimes e_{i+N}^{2N}\right)\Sigma_{\sigma}(Z)+t\sum_{i=1}^{N}e_{i}^{2N}\otimes e_{i+N}^{2N}Z & \quad \text{if}\hspace{2mm}Z\in\mathcal{M}_{t}^{+}
\end{array}
\right.
\end{equation*}
for $Z\in\mathcal{M}$ and any $t\in\mathbb{R}$. As such, the flow map $T^{t}_{\sigma}$ acts as a \emph{shear} operator on $\mathcal{M}_{t}^{-}$ and its action is independent of the choice of $N$-body scattering map $\sigma$. For any $\Phi\in C^{0}_{c}(\mathbb{R}^{2N}, \mathbb{R})$ with the property that
\begin{equation*}
\mathscr{H}(\mathrm{supp}(\Phi)\cap\mathcal{M}_{t}^{-})>0,    
\end{equation*}
we find that 
\begin{equation}\label{refme}
\langle T^{t}_{\sigma}\#\mu, \Phi \rangle=\int_{\mathcal{M}}\Phi\,d(T^{t}_{\sigma}\#\mu)=\int_{T^{-t}_{\sigma}\circ\Psi(\mathcal{U}_{t}^{-})}(\Phi\circ T^{t}_{\sigma}\circ\Psi)(M\circ\Psi)\omega \,d\mathscr{L}.  
\end{equation}
In order to be able to process the information in the integral \eqref{refme} above, for any $N\geq 3$, we require the following Lemma on the explicit form of the surface density $\omega:\Omega\rightarrow [0, \infty)$ associated to the chart map $\Psi$ defined above.
\begin{lem}
For any $N\geq 3$, it holds that the surface density $\omega$ associated to the chart map $\Psi$ is given by
\begin{equation}\label{surfacedensity}
\omega(X, U)=\frac{((u_{1}-u_{2})^{2}+(x_{1}-x_{2})^{2})^{\frac{N-2}{2}}}{(x_{1}-x_{2})^{N-1}}\left(\sum_{i>j}(x_{i}-x_{j})^{2}\right)^{\frac{1}{2}}
\end{equation}
for $(X, U)\in\Omega$.
\end{lem}
\begin{proof}
To produce the general formula \eqref{surfacedensity} above, we evaluate the determinant of $D\Psi^{T} D\Psi$ for any $N\geq 3$ in a systematic manner. Indeed, employing elementary row and column operation matrices, we show that $D\Psi^{T}D\Psi$ is equivalent to a lower-triangular matrix-valued map, following which the evaluation of its determinant is trivial. 

To begin, it can be shown that the matrix-valued map $D\Psi:\Omega\rightarrow\mathbb{R}^{2N\times(N+2)}$ is given block-wise by
\begin{equation*}
D\Psi(X, U)=\left(
\begin{array}{ccc}
 & I_{N+2} & \vspace{2mm} \\ \hline \vspace{-3mm} \\
 A(X, U) & \displaystyle\left(\frac{u_{1}-u_{2}}{x_{1}-x_{2}}\right)I_{N-2} & B(X, U)
\end{array}
\right),
\end{equation*}
where $A:\Omega\rightarrow\mathbb{R}^{(N-2)\times 2}$ is given by
\begin{equation*}
A_{j, 1}(X, U):=\frac{(u_{1}-u_{2})(x_{2}-x_{j+2})}{(x_{1}-x_{2})^{2}}
\end{equation*}
and
\begin{equation*}
A_{j, 2}(X, U):=-\frac{(u_{1}-u_{2})(x_{1}-x_{j+2})}{(x_{1}-x_{2})^{2}}
\end{equation*}
for $j\in\{1, ..., N-2\}$, while $B:\Omega\rightarrow\mathbb{R}^{(N-2)\times 2}$ is given by
\begin{equation*}
B_{j, 1}:=-\frac{x_{2}-x_{j+2}}{x_{1}-x_{2}}
\end{equation*}
and
\begin{equation*}
B_{j, 2}:=\frac{x_{1}-x_{j+2}}{x_{1}-x_{2}}
\end{equation*}
for $j\in\{1, ..., N-2\}$. In turn, it follows that the product map $D\Psi^{T} D\Psi:\Omega\rightarrow\mathbb{R}^{(N+2)\times (N+2)}$ is given block-wise by 
\begin{equation}\label{toughmatrix}
D\Psi(X, U)^{T}D\Psi(X, U)=
\left(
\begin{array}{ccc}
\alpha(X, U) & \gamma(X, U)^{T} & \delta(X, U) \vspace{2mm}\\
\gamma(X, U) & \displaystyle \left(1+\frac{(u_{1}-u_{2})^{2}}{(x_{1}-x_{2})^{2}}\right)I_{N-2} & \tau(X, U)\gamma(X, U) \vspace{2mm} \\
\delta(X, U) & \tau(X, U)\gamma(X, U)^{T} & \beta(X, U)
\end{array}
\right),
\end{equation}
where the block map $\alpha:\Omega\rightarrow\mathbb{R}^{2\times 2}$ is given by
\begin{equation*}
\alpha_{1, 1}(X, U):=1+\frac{(u_{1}-u_{2})^{2}}{(x_{1}-x_{2})^{4}}\sum_{j=3}^{N}(x_{2}-x_{j})^{2},    
\end{equation*}
\begin{equation*}
\alpha_{1, 2}(X, U)=\alpha_{2, 1}(X, U):=-\frac{(u_{1}-u_{2})^{2}}{(x_{1}-x_{2})^{4}}\sum_{j=3}^{N}(x_{1}-x_{j})(x_{2}-x_{j}),    
\end{equation*}
\begin{equation*}
\alpha_{2, 2}(X, U):=1+\frac{(u_{1}-u_{2})^{2}}{(x_{1}-x_{2})^{4}}\sum_{j=3}^{N}(x_{1}-x_{j})^{2},    
\end{equation*}
with $\beta:\Omega\rightarrow\mathbb{R}^{2\times 2}$ given by
\begin{equation*}
\beta_{1, 1}(X, U):=1+\frac{1}{(x_{1}-x_{2})^{2}}\sum_{j=3}^{N}(x_{2}-x_{j})^{2},    
\end{equation*}
\begin{equation*}
\beta_{1, 2}(X, U)=\beta_{2, 1}(X, U):=-\frac{1}{(x_{1}-x_{2})^{2}}\sum_{j=3}^{N}(x_{1}-x_{j})(x_{2}-x_{j}),    
\end{equation*}
\begin{equation*}
\beta_{2, 2}(X, U):=1+\frac{1}{(x_{1}-x_{2})^{2}}\sum_{j=3}^{N}(x_{1}-x_{j})^{2}    
\end{equation*}
for $(X, U)\in\Omega$. Moreover, $\gamma:\Omega\rightarrow\mathbb{R}^{(N-2)\times 2}$ is defined by
\begin{equation*}
\gamma_{j, 1}(X, U):=\frac{(u_{1}-u_{2})^{2}}{(x_{1}-x_{2})^{3}}(x_{2}-x_{j+2})    
\end{equation*}
and
\begin{equation*}
\gamma_{j, 2}(X, U):=-\frac{(u_{1}-u_{2})^{2}}{(x_{1}-x_{2})^{3}}(x_{1}-x_{j+2})    
\end{equation*}
for $j\in\{1, ..., N-2\}$, while $\delta:\Omega\rightarrow\mathbb{R}^{2\times 2}$ is defined by
\begin{equation*}
\delta_{1, 1}(X, U):=-\frac{u_{1}-u_{2}}{(x_{1}-x_{2})^{3}}\sum_{j=3}^{N}(x_{2}-x_{j})^{2},
\end{equation*}
\begin{equation*}
\delta_{1, 2}(X, U)=\delta_{2, 1}(X, U):=\frac{u_{1}-u_{2}}{(x_{1}-x_{2})^{3}}\sum_{j=3}^{N}(x_{1}-x_{j})(x_{2}-x_{j}),    
\end{equation*}
\begin{equation*}
\delta_{2, 2}(X, U):=-\frac{u_{1}-u_{2}}{(x_{1}-x_{2})^{3}}\sum_{j=3}^{N}(x_{1}-x_{j})^{2}    
\end{equation*}
for $(X, U)\in\Omega$. We now claim that for each $(X, U)\in\Omega$ the matrix \eqref{toughmatrix} can be reduced to lower-triangular form by both elementary row and column operations. Indeed, if we define the elementary matrix $\mathsf{E}(i, j, a)\in\mathbb{R}^{(N+2)\times (N+2)}$ by
\begin{equation*}
\mathsf{E}(i, j, a):=I_{N+2}+a(e_{i}^{N+2}\otimes e_{j}^{N+2})
\end{equation*}
for $i, j\in\{1, ..., N+2\}$ and $a\in\mathbb{R}$, then it holds that the product of matrices given by
\begin{equation}\label{lowerdiagid}
\begin{array}{c}
\displaystyle \left(\prod_{k=1}^{2}\mathsf{E}\left(k, N+k, \frac{u_{1}-u_{2}}{x_{1}-x_{2}}\right)\right)D\Psi(X, U)^{T}D\Psi(X, U)\left(\prod_{k=1}^{2}\mathsf{E}\left(k, N+k, -\frac{u_{1}-u_{2}}{x_{1}-x_{2}}\right)\right)\times\vspace{2mm}\\
\displaystyle \prod_{k=3}^{N}\mathsf{E}\left(k, N+1, \frac{(u_{1}-u_{2})(x_{2}-x_{k})}{(x_{1}-x_{2})^{2}}\right)\prod_{k=3}^{N}\mathsf{E}\left(k, N+2, -\frac{(u_{1}-u_{2})(x_{1}-x_{k})}{(x_{1}-x_{2})^{2}}\right)\times\vspace{2mm}\\\displaystyle
\mathsf{E}\left(N+1, N+2, \left(\sum_{j=1}^{N}(x_{2}-x_{j})^{2}\right)^{-1}\sum_{k=3}^{N}(x_{1}-x_{k})(x_{2}-x_{k})\right)
\end{array}
\end{equation}
equals the lower-triangular matrix 
\begin{equation*}   
\displaystyle 
\left(
\begin{array}{ccccccccc}
1 & 0 & 0 & 0 & \dots & 0 & 0 & 0 & 0 \vspace{2mm}\\
\star & 1 & 0 & 0 & \dots & 0 & 0 & 0 & 0 \vspace{2mm}\\
\star & \star & 1+\frac{(u_{1}-u_{2})^{2}}{(x_{1}-x_{2})^{2}} & 0 & \dots & 0 & 0 & 0 & 0 \vspace{2mm}\\
\star & \star & 0 & 1+\frac{(u_{1}-u_{2})^{2}}{(x_{1}-x_{2})^{2}} & \dots & 0 & 0 & 0 & 0\vspace{2mm}\\
\vdots & \vdots & \vdots & \vdots & \ddots & \vdots & \vdots & \vdots & \vdots \vspace{2mm} \\
\star & \star & 0 & 0 & \dots & 1+\frac{(u_{1}-u_{2})^{2}}{(x_{1}-x_{2})^{2}} & 0 & 0 & 0\vspace{2mm}\\
\star & \star & 0 & 0 & \dots & 0 & 1+\frac{(u_{1}-u_{2})^{2}}{(x_{1}-x_{2})^{2}} & 0 & 0 \vspace{2mm} \\
\star & \star & \star & \star & \dots & \star & \star & p(\zeta) & 0 \vspace{2mm}\\
\star & \star & \star & \star & \dots & \star & \star & \star & q(\zeta) 
\end{array}
\right),
\end{equation*}
where $\star$ denotes an entry whose value is immaterial to the value of the determinant, and where $p, q:\Omega\rightarrow\mathbb{R}$ are defined by
\begin{equation*}
p(\zeta):=\frac{1}{(x_{1}-x_{2})^{2}}\sum_{j=1}^{N}(x_{2}-x_{j})^{2}   
\end{equation*}
and
\begin{equation*}
q(\zeta):=\left(\sum_{j=1}^{N}(x_{2}-x_{j})^{2}\right)^{-1}\sum_{i>j}(x_{i}-x_{j})^{2}    
\end{equation*}
for $\zeta\in\Omega$. As the determinant of a lower-triangular matrix is given by the product of the elements on the diagonal, we find from \eqref{lowerdiagid} above that
\begin{equation}\label{almostthere}
\begin{array}{c}
\displaystyle \prod_{k=1}^{2}\mathrm{det}\,\mathsf{E}\left(k, N+k, \frac{u_{1}-u_{2}}{x_{1}-x_{2}}\right)\mathrm{det}\,\mathsf{E}\left(k, N+k, -\frac{u_{1}-u_{2}}{x_{1}-x_{2}}\right)\times \vspace{2mm} \\  
\displaystyle \prod_{k=3}^{N}\mathrm{det}\,\mathsf{E}\left(k, N+1, \frac{(u_{1}-u_{2})(x_{2}-x_{k})}{(x_{1}-x_{2})^{2}}\right)\mathrm{det}\,\mathsf{E}\left(k, N+2, -\frac{(u_{1}-u_{2})(x_{1}-x_{k})}{(x_{1}-x_{2})^{2}}\right)\times \vspace{2mm} \\
\displaystyle \mathrm{det}\,\mathsf{E}\left(N+1, N+2, \left(\sum_{j=1}^{N}(x_{2}-x_{j})^{2}\right)^{-1}\sum_{k=3}^{N}(x_{1}-x_{k})(x_{2}-x_{k})\right)\mathrm{det}(D\Psi(\zeta)^{T}D\Psi(\zeta))=\vspace{2mm}\\
\displaystyle \frac{((u_{1}-u_{2})^{2}+(x_{1}-x_{2})^{2})^{N-2}}{(x_{1}-x_{2})^{2(N-1)}}\sum_{i>j}(x_{i}-x_{j})^{2}.
\end{array}
\end{equation}
Finally, as it holds that
\begin{equation*}
\mathrm{det}(\mathsf{E}(i, j, a))=1    
\end{equation*}
for any $i, j\in\{1, ..., N+2\}$ and any $a\in\mathbb{R}$, we infer from \eqref{almostthere} above that
\begin{equation*}
\mathrm{det}(D\Psi(\zeta)^{T}D\Psi(\zeta))=\frac{((u_{1}-u_{2})^{2}+(x_{1}-x_{2})^{2})^{N-2}}{(x_{1}-x_{2})^{2(N-1)}}\sum_{i>j}(x_{i}-x_{j})^{2}    
\end{equation*}
for all $\zeta\in \Omega$, from which follows the proof of the Proposition.
\end{proof}
\subsection{Some Auxiliary Maps and Identities}\label{auxmappies}
It is not straightforward to perform a change of coordinates in the integral
\begin{equation}\label{inttofind}
\int_{\Omega}\Phi(T^{t}_{\sigma}(\Psi(\zeta)))M(\Psi(\zeta))\omega(\zeta)\,d\mathscr{L}(\zeta)    
\end{equation}
with the aim of eliminating the presence of compositions in the argument of $\Phi$, since the dimension of the domain of the map $T^{t}_{\sigma}\circ\Psi$ as a subset of Euclidean space does not match the dimension of its range. Given our choice of chart map $\Psi$, it is prudent to define a \emph{reduced flow map} $\widetilde{T}^{t}$ pointwise by
\begin{equation*}
\widetilde{T}^{t}(\zeta):=\zeta+t\sum_{i=1}^{N}(e_{i}^{N+2}\otimes e_{i+N}^{2N})\Psi(\zeta)
\end{equation*}
for all $\zeta=(X, U)\in\Omega$ and each $t\in\mathbb{R}$. Evidently, this reduced flow map has the property that the dimension of its domain is the same as its range. The reader may check that, for each $t\in\mathbb{R}$, the map $\widetilde{T}^{t}$ is a smooth diffeomorphism of $\mathcal{U}_{t}^{-}$ with inverse given by $(\widetilde{T}^{t})^{-1}=\widetilde{T}^{-t}$. However, from the point of view of \eqref{inttofind}, we must understand how the remaining $N-2$ `excluded' velocity variables transform when composed with this map. Indeed, for each $N\geq 3$ we define the \emph{reduced velocity map} $\overline{\psi}:\mathcal{U}_{t}^{-}\rightarrow\mathbb{R}^{N-2}$ pointwise by
\begin{equation*}
\overline{\psi}_{i}(\zeta):=\psi_{i+2}(\zeta)
\end{equation*}
for $\zeta\in\mathcal{U}_{t}^{-}$ and $i\in\{1, ..., N-2\}$. As such, we require the following Lemma, which we state without proof.
\begin{lem}
For $N\geq 3$, it holds that
\begin{equation*}
(\overline{\psi}\circ \widetilde{T}^{-t})(\zeta)=\overline{\psi}(\zeta)    
\end{equation*}
for all $\zeta\in\widetilde{T}^{-t}(\mathcal{U}_{t}^{-})$ and all $t\in\mathbb{R}$. 
\end{lem}
As a consequence of this Lemma, it follows that
\begin{equation}\label{usefulid}
\Phi(\zeta, \overline{\psi}(\widetilde{T}^{-t}(\zeta)))=\Phi(\zeta, \overline{\psi}(\zeta))=\Phi(\Psi(\zeta))    
\end{equation}
for all $t\in\mathbb{R}$, which is an identity of which we make use in the sequel. We also require the following formula for the Jacobian of the coordinate transformation $\widetilde{T}^{t}$ which holds for any $N\geq 3$.
\begin{lem}
For any $N\geq 3$, it holds that
\begin{equation}\label{translatejaccy}
\mathrm{det}(D\widetilde{T}^{t}(\zeta))=\left(\frac{x_{1}+tu_{1}-x_{2}-tu_{2}}{x_{1}-x_{2}}\right)^{N-2}
\end{equation}
for each $\zeta=(X, U)\in(\mathcal{U}_{t}^{-})^{\circ}$ and each $t\in\mathbb{R}$.
\end{lem}
\begin{proof}
It is straightforward to show by differentiation that $D\widetilde{T}^{t}:\mathcal{U}_{t}^{-}\rightarrow\mathbb{R}^{(N+2)\times (N+2)}$ is given blockwise by
\begin{equation*}
D\widetilde{T}^{t}(X, U)=\left(
\begin{array}{c  c  c}
I_{2} & 0_{2\times (N-2)} & tI_{2} \vspace{2mm} \\
A(X, U, t) & \displaystyle \left(\frac{x_{1}+tu_{1}-x_{2}-tu_{2}}{x_{1}-x_{2}}\right)I_{N-2} & B(X, U, t) \vspace{2mm}\\
0_{2} & 0_{2\times (N-2)} & I_{2}
\end{array}
\right),    
\end{equation*}
where $A:\Omega\times\mathbb{R}\rightarrow \mathbb{R}^{(N-2)\times 2}$ is given by
\begin{equation*}
A_{j, 1}(X, U, t):=\frac{t(u_{1}-u_{2})(x_{2}-x_{j+2})}{(x_{1}-x_{2})^{2}}    
\end{equation*}
and
\begin{equation*}
A_{j, 2}(X, U, t):=-\frac{t(u_{1}-u_{2})(x_{1}-x_{j+2})}{(x_{1}-x_{2})^{2}}
\end{equation*}
for $j\in\{1, ..., N-2\}$, and $B:\Omega\times\mathbb{R}\rightarrow \mathbb{R}^{(N-2)\times 2}$ is given by
\begin{equation*}
B_{j, 1}(X, U, t):=-\frac{t(x_{2}-x_{j+2})}{x_{1}-x_{2}}    
\end{equation*}
and 
\begin{equation*}
B_{j, 2}(X, U, t):=\frac{t(x_{1}-x_{j+2})}{x_{1}-x_{2}}        
\end{equation*}
for $j\in\{1, ..., N-2\}$. By performing elementary column operations, it holds that the product
\begin{equation*}
\begin{array}{c}
\displaystyle D\widetilde{T}^{t}(\zeta)\prod_{i=1}^{2}\mathsf{E}(i, N+i, -t)\prod_{j=3}^{N}\mathsf{E}\left(j, N+1, t\frac{x_{2}-x_{j}}{x_{1}-x_{2}}\right)\times\vspace{2mm} \\
\displaystyle \prod_{k=3}^{N}\mathsf{E}\left(k, N+2, -t\frac{x_{1}-x_{k}}{x_{1}-x_{2}}\right)
\end{array}
\end{equation*}
equals the lower-triangular matrix
\begin{equation*}
\left(
\begin{array}{ccccccccc}
1 & 0 & 0 & 0 & \dots & 0 & 0 & 0 & 0 \\
0 & 1 & 0 & 0 & \dots & 0 & 0 & 0 & 0 \\
\star & \star & \frac{x_{1}+tu_{1}-x_{2}-tu_{2}}{x_{1}-x_{2}} & 0 & \dots & 0 & 0 & 0 & 0 \vspace{1mm}\\ 
\star & \star & 0 & \frac{x_{1}+tu_{1}-x_{2}-tu_{2}}{x_{1}-x_{2}} & \dots & 0 & 0 & 0 & 0 \vspace{1mm} \\
\vdots & \vdots & \vdots & \vdots & \ddots & \vdots & \vdots & \vdots & \vdots \\
\star & \star & 0 & 0 & \dots & \frac{x_{1}+tu_{1}-x_{2}-tu_{2}}{x_{1}-x_{2}} & 0 & 0 & 0 \vspace{1mm} \\
\star & \star & 0 & 0 & \dots & 0 & \frac{x_{1}+tu_{1}-x_{2}-tu_{2}}{x_{1}-x_{2}} & 0 & 0 \vspace{1mm} \\
0 & 0 & 0 & 0 & \dots & 0 & 0 & 1 & 0 \\
0 & 0 & 0 & 0 & \dots & 0 & 0 & 0 & 1 \\
\end{array}
\right)
\end{equation*}
for $\zeta\in\mathcal{U}_{t}^{-}$. Once again, as the determinant of a lower-triangular matrix equals the product of the elements on its diagonal, it follows that
\begin{equation*}
\mathrm{det}(D\widetilde{T}^{t}(X, U))=\left(\frac{x_{1}+tu_{1}-x_{2}-tu_{2}}{x_{1}-x_{2}}\right)^{N-2}    
\end{equation*}
as claimed in the statement of the Lemma.
\end{proof}
We also require the following structural result, whose algebraic proof we also leave to the reader.
\begin{lem}\label{densitycomp1}
For any $N\geq 3$, it holds that
\begin{equation*}
\begin{array}{c}
(\omega\circ \widetilde{T}^{-t})(\zeta)= \vspace{2mm} \\ 
\displaystyle \left(\frac{x_{1}-x_{2}}{x_{1}-tu_{1}-x_{2}+tu_{2}}\right)^{N-2}\left(\frac{(u_{1}-u_{2})^{2}+(x_{1}-tu_{1}-x_{1}+tu_{2})}{(u_{1}-x_{2})^{2}+(x_{1}-x_{2})^{2}}\right)^{\frac{N-2}{2}}\omega(\zeta)    
\end{array}
\end{equation*}
for $\zeta=(X, U)\in\widetilde{T}^{-t}(\mathcal{U}_{t}^{-})$ and $t\in\mathbb{R}$.
\end{lem}
\subsection{Derivation of the Density Functional Equation}\label{derifunky}
With the above building blocks of Section \ref{auxmappies} in place, we may now approach the derivation of a functional equation which, if satisfied by a density $M$, generates an invariant measure $\mu=M\lambda$. For any $\Phi\in C^{0}_{c}(\mathbb{R}^{2N}, \mathbb{R})$ with the support property $\mathscr{H}(\mathrm{supp}(\Phi)\cap\mathcal{M}_{t}^{-})>0$, we have that
\begin{equation*}
\begin{array}{cl}
 & \langle T^{t}_{\sigma}\#\mu, \Phi\rangle \vspace{2mm}\\
 = & \displaystyle \int_{\widetilde{T}^{-t}(\mathcal{U}_{t}^{-})}\Phi(\widetilde{T}^{t}(\zeta), \overline{\psi}(\zeta))M(\zeta, \overline{\psi}(\zeta))\omega(\zeta)\,d\mathscr{L}(\zeta) \vspace{2mm} \\
 = & \displaystyle \int_{\mathcal{U}_{t}^{-}}\Phi(\zeta, \overline{\psi}(\widetilde{T}^{-t}(\zeta)))M(\widetilde{T}^{-t}(\zeta), \overline{\psi}(\widetilde{T}^{-t}(\zeta)))\omega(\widetilde{T}^{-t}(\zeta))\left|\mathrm{det}(D\widetilde{T}^{t}(\widetilde{T}^{-t}(\zeta)))^{-1}\right|\,d\mathscr{L}(\zeta) \vspace{2mm} \\
 \overset{\eqref{translatejaccy}}{=} & \displaystyle \int_{\mathcal{U}_{t}^{-}}\Phi(\zeta, \overline{\psi}(\widetilde{T}^{-t}(\zeta)))M(\widetilde{T}^{-t}(\zeta), \overline{\psi}(\widetilde{T}^{-t}(\zeta)))\omega(\widetilde{T}^{-t}(\zeta))\left|\frac{x_{1}-tu_{1}-x_{2}+tu_{2}}{x_{1}-x_{2}}\right|^{N-2}\,d\mathscr{L}(\zeta)  \vspace{2mm} \\
 \overset{\eqref{usefulid}}{=} & \displaystyle\int_{\mathcal{U}_{t}^{-}}\Phi(\Psi(\zeta))M(\widetilde{T}^{-t}(\zeta), \overline{\psi}(\zeta))\omega(\widetilde{T}^{-t}(\zeta))\left|\frac{x_{1}-tu_{1}-x_{2}+tu_{2}}{x_{1}-x_{2}}\right|^{N-2}\,d\mathscr{L}(\zeta),
\end{array}
\end{equation*}
and as by Lemma \ref{densitycomp1} 
\begin{equation*}
\omega(\widetilde{T}^{-t}(\zeta))\left(\frac{x_{1}-tu_{1}-x_{2}+tu_{2}}{x_{1}-x_{2}}\right)^{N-2}= \left(\frac{(u_{1}-u_{2})^{2}+(x_{1}-tu_{1}-x_{1}+tu_{2})^{2}}{(u_{1}-x_{2})^{2}+(x_{1}-x_{2})^{2}}\right)^{\frac{N-2}{2}}\omega(\zeta)    
\end{equation*}
for all $\zeta\in\Omega$, we conclude that
\begin{equation*}
\int_{\mathcal{M}}\Phi\,d(T^{t}_{\sigma}\#\mu)=\int_{\mathcal{M}}\Phi\,d\mu   
\end{equation*}
for all $\Phi\in C^{0}_{c}(\mathbb{R}^{2N}, \mathbb{R})$ with the aforementioned support property if the density $M$ of the measure $\mu$ satisfies the identity
\begin{equation*}
\left(\frac{(u_{1}-u_{2})^{2}+(x_{1}-tu_{1}-x_{1}+tu_{2})^{2}}{(u_{1}-u_{2})^{2}+(x_{1}-x_{2})^{2}}\right)^{\frac{N-2}{2}}M(\widetilde{T}^{-t}(\zeta), \overline{\psi}(\zeta))=M(\Psi(\zeta)),    
\end{equation*}
which reads more straightforwardly as 
\begin{equation}\label{firstmid}
\left(\frac{(u_{1}-u_{2})^{2}+(x_{1}-tu_{1}-x_{1}+tu_{2})^{2}}{(u_{1}-u_{2})^{2}+(x_{1}-x_{2})^{2}}\right)^{\frac{N-2}{2}}M(X-t\psi(\zeta), \psi(\zeta))=M(X, \psi(\zeta))    
\end{equation}
for all $\zeta=(X, U)\in\Omega$ and all $t\in\mathbb{R}$. The reader will note that there appears to be a spurious bias on the variables indexed by 1 and 2 in the above identity \eqref{firstmid}. However, this is simply a consequence of our choice of chart map $\Psi$. Indeed, we may rewrite identity \eqref{firstmid} in such a way it becomes symmetric in all variables. In this direction, we require another structural lemma which we state without proof.
\begin{lem}\label{ratioids}
For any $N\geq 3$, it holds that
\begin{equation*}
\begin{array}{c}
\displaystyle \frac{(u_{1}-u_{2})^{2}+(x_{1}-tu_{1}-x_{1}+tu_{2})^{2}}{(u_{1}-u_{2})^{2}+(x_{1}-x_{2})^{2}}= \vspace{2mm} \\
\displaystyle \frac{(\psi_{i}(\zeta)-\psi_{j}(\zeta))^{2}+(x_{i}-t\psi_{i}(\zeta)-x_{j}+t\psi_{j}(\zeta))^{2}}{(\psi_{i}(\zeta)-\psi_{j}(\zeta))^{2}+(x_{i}-x_{j})^{2}} 
\end{array}
\end{equation*}
for $i\in\{1, ..., N\}$ with $i<j$ and $\zeta\in\mathcal{U}_{t}^{-}$.
\end{lem}
Employing the result of Lemma \ref{ratioids} above, the identity \eqref{firstmid} may be rewritten in the more symmetric form as
\begin{equation*}
\begin{array}{c}
\displaystyle \left(\prod_{i<j}\frac{(\psi_{i}-\psi_{j})^{2}+(x_{i}-t\psi_{i}-x_{j}+t\psi_{j})^{2}}{(\psi_{i}-\psi_{j})^{2}+(x_{i}-x_{j})^{2}}\right)^{\frac{N-2}{N(N-1)}}M(X-t\psi(X, U), \psi(X, U)) = \vspace{2mm} \\
M(X, \psi(X, U)),
\end{array}
\end{equation*}
whence, owing to the fact that $\Psi$ is a homeomorphism, it holds that
\begin{equation}\label{readytobesimplified}
\left(\prod_{i<j}\frac{(v_{i}-v_{j})^{2}+(x_{i}-tv_{i}-x_{j}+tv_{j})^{2}}{(v_{i}-v_{j})^{2}+(x_{i}-x_{j})^{2}}\right)^{\frac{N-2}{N(N-1)}}M(X-tV, V)=M(X, V)
\end{equation}
for all $V\in F_{X}$ and $X\in\mathcal{Q}^{\circ}$. It is at this point we acknowledge that the Hausdorff measure $\mathscr{H}$ cannot be an invariant measure for any billiard flow on $\mathcal{M}$. We record this observation in the following Lemma. 
\begin{lem}
Let $\{T^{t}\}_{t\in\mathbb{R}}$ be a billiard flow on $\mathcal{M}$. For $t\in\mathbb{R}\setminus\{0\}$, it holds that $T^{t}\#\mathscr{H}\neq \mathscr{H}$.    
\end{lem}
\begin{proof}
This follows from identity \eqref{readytobesimplified} above by setting the density $M$ to be unity.  
\end{proof}
It is now possible to simplify the structure of the functional equation \eqref{readytobesimplified} above. Indeed, defining a new measurable map $m:\sqcup_{X\in\mathcal{Q}^{\circ}}F_{X}\rightarrow\mathbb{R}$ pointwise by
\begin{equation*}
m(X, V):=\left(\prod_{i<j}(v_{i}-v_{j})^{2}+(x_{i}-x_{j})^{2}\right)^{\frac{N-2}{N(N-1)}}M(X, V),    
\end{equation*}
it holds that $\langle T^{t}_{\sigma}\#\mu, \Phi\rangle=\langle\mu, \Phi\rangle$ for all $\Phi\in C^{0}_{c}(\mathbb{R}^{2N}, \mathbb{R})$ admitting the support condition $\mathscr{H}(\mathrm{supp}(\Phi)\cap\mathcal{M}_{t}^{-})>0$ if $m$ satisfies the functional equation
\begin{equation}\label{m0funky}
m(X-tV, V)=m(X, V)    
\end{equation}
for all $V\in F_{X}$ and all $t\in\mathbb{R}$. We shall solve this functional equation in the following Section.
\subsection{Solving the Functional Equation}\label{solvefunky}
It may appear at first glance that the functional equation \eqref{m0funky} is trivial. However, this is not the case due to the algebraic relationships between $X$ and $V$, owing to the fact that $Z=(X, V)$ lies in a non-trivial algebraic subvariety of $\mathbb{R}^{2N}$. As we do not seek to characterise all invariant measures which are absolutely continuous with respect to the Liouville measure $\lambda$, rather seeking only sufficient conditions on their densities, we assume in what follows that the density $M$ (and, in turn, $m$) is defined on the whole space $\mathbb{R}^{2N}$ as opposed to simply on the subset $\sqcup_{X\in\mathcal{Q}^{\circ}}F_{X}$. We now record the following Proposition.
\begin{prop}
Let $N\geq 3$. Suppose that $m:\mathbb{R}^{2N}\rightarrow\mathbb{R}$ is a measurable map satisfying the functional equation
\begin{equation*}
m(X-tV, V)=m(X, V)    
\end{equation*}
for all $V\in F_{X}$ and all $X\in\mathcal{Q}^{\circ}$. It follows that $m$ is necessarily of the form
\begin{equation*}
m(X, V)=m_{0}(X\cdot (\mathsf{E}(V)\mathbf{1}-\mathsf{P}(V)V), V)    
\end{equation*}
for some measurable map $m_{0}:\mathbb{R}\times\mathbb{R}^{N}\rightarrow\mathbb{R}$ and for all $V\in F_{X}$ and $X\in\mathcal{Q}^{\circ}$, where $\mathsf{P}$ is the linear momentum map \eqref{linmommap} and $\mathsf{E}$ is the energy map \eqref{enmap} above.
\end{prop}
\begin{proof}
We begin by noting that $V\in F_{X}$ if and only if $X\in G_{V}$, where $G_{V}\subset\mathcal{Q}$ is given by
\begin{equation*}
G_{V}:=\left\{y_{1}\left(e_{1}^{N}+\sum_{i=1}^{3}\frac{v_{i}-v_{2}}{v_{1}-v_{2}}e_{i}^{N}\right)+y_{2}\left(e_{2}^{N}-\sum_{j=3}^{N}\frac{v_{j}-v_{1}}{v_{1}-v_{2}}e_{j}^{N}\right)\,:\,y_{1}\leq y_{2}\right\}.    
\end{equation*}
As a consequence, the invariant manifold $\mathcal{M}$ may be written in bundle notation by `switching' the role of the spatial and velocity variables as
\begin{equation*}
\mathcal{M}=\bigsqcup_{V\in\mathcal{V}}G_{V},    
\end{equation*}
where $\mathcal{V}:=\mathcal{V}^{-}\cup\mathcal{V}^{+}$. Manifestly, $G_{V}$ is a subset of the 2-dimensional real vector space $\mathcal{G}_{V}$ given by
\begin{equation*}
\mathcal{G}_{V}:=\mathrm{span}\left\{
e_{1}^{N}+\sum_{i=1}^{3}\frac{v_{i}-v_{2}}{v_{1}-v_{2}}e_{i}^{N}, e_{2}^{N}-\sum_{i=3}^{N}\frac{v_{i}-v_{1}}{v_{1}-v_{2}}e_{i}^{N}\right\}\subset\mathbb{R}^{N}.    
\end{equation*}
Indeed, let us subsequently denote the vectors in the above span by
\begin{equation*}
\gamma_{1}:=e_{1}^{N}+\sum_{i=1}^{3}\frac{v_{i}-v_{2}}{v_{1}-v_{2}}e_{i}^{N}    
\end{equation*}
and 
\begin{equation*}
\gamma_{2}:=e_{2}^{N}-\sum_{i=3}^{N}\frac{v_{i}-v_{1}}{v_{1}-v_{2}}e_{i}^{N}.    
\end{equation*}
Let $V\in\mathcal{V}$ be fixed. By defining the measurable auxiliary map $m_{V}:\mathcal{Q}\rightarrow\mathbb{R}$ pointwise by
\begin{equation*}
m_{V}(X):=m(X, V),    
\end{equation*}
the functional equation may be rewritten as
\begin{equation}\label{rewrite}
m_{V}(X-tV)=m_{V}(X).    
\end{equation}
As a result, we interpret the functional equation rather as a $V$-parametrised functional equation for an unknown function $m_{V}$ defined on $G_{V}$. For this given $V$, it holds that if $X\in G_{V}$ then $X-tV\in G_{V}$ for all suitable $t\in\mathbb{R}$. Thus, the value of $m_{V}(X)$ depends only on the component of $X$ which lies in the orthogonal complement of $\mathrm{span}\{V\}$ in $\mathcal{G}_{V}$.

It holds that $V\in G_{V}$ owing to the fact that $V=v_{1}\gamma_{1}+v_{2}\gamma_{2}$. By employing the Gram-Schmidt algorithm to the basis $\{v_{1}\gamma_{1}+v_{2}\gamma_{2}, v_{1}\gamma_{1}-v_{2}\gamma_{2}\}$ of $\mathcal{G}_{V}$, the reader may verify that
\begin{equation*}
\mathcal{G}_{V}=\mathrm{span}\left\{V, \mathsf{E}(V)\mathbf{1}-\mathsf{P}(V)V\right\}.    
\end{equation*}
As such, given $X\in \mathcal{G}_{V}$, it holds that 
\begin{equation}\label{basisexpansion}
X=aV+b(\mathsf{E}(V)\mathbf{1}-\mathsf{P}(V)V)
\end{equation}
for some $a, b\in\mathbb{R}$ which depend on $X$. Thus, it follows from the $V$-parametrised functional equation \eqref{rewrite} that
\begin{equation*}
m_{V}(aV+b(\mathsf{E}(V)\mathbf{1}-\mathsf{P}(V)V))=m_{V}((a-t)V+b(\mathsf{E}(V)\mathbf{1}-\mathsf{P}(V)V))    
\end{equation*}
for all $t$. However, by taking $t=a$, it follows immediately from the above that
\begin{equation}\label{finalinfer}
m_{V}(X)=m_{V}(b(\mathsf{E}(V)\mathbf{1}-\mathsf{P}(V)V)).    
\end{equation}
From the identity \eqref{basisexpansion} above, it follows that $b$ is given explictly in terms of $X$ by
\begin{equation*}
b=\frac{X\cdot (\mathsf{E}(V)\mathbf{1}-\mathsf{P}(V)V)}{|\mathsf{E}(V)\mathbf{1}-\mathsf{P}(V)V|^{2}}.    
\end{equation*}
Thus, defining the measurable function $m_{0}:\mathbb{R}\times\mathcal{V}\rightarrow\mathbb{R}$ pointwise by
\begin{equation*}
m_{0}(x, V):=m_{V}\left(\frac{x}{|\mathsf{E}(V)\mathbf{1}-\mathsf{P}(V)V|^{2}}\left(\mathsf{E}(V)\mathbf{1}-\mathsf{P}(V)V\right)\right)  
\end{equation*}
for all $x\in\mathbb{R}$ and $V\in\mathcal{V}$, it follows in turn from \eqref{finalinfer} that
\begin{equation*}
m(X, V)=m_{0}(X\cdot (\mathsf{E}(V)\mathbf{1}-\mathsf{P}(V)V), V)   
\end{equation*}
as claimed in the statement of the Proposition.
\end{proof}
We now have that the measure $\mu=M\mathscr{H}$ satisfies the identity
\begin{equation*}
\langle T^{t}_{\sigma}\#\mu, \Phi \rangle=\langle\mu, \Phi\rangle    
\end{equation*}
for all $\Phi$ satisfying the support condition $\mathscr{H}(\mathrm{supp}(\Phi)\cap\mathcal{M}_{t}^{-})>0$ if the density $M$ is of the form
\begin{equation*}
M(X, V)=\left(\prod_{i<j}((v_{i}-v_{j})^{2}+(x_{i}-x_{j})^{2})\right)^{-\frac{N-2}{N(N-1)}}m_{0}(X\cdot (\mathsf{E}(V)\mathbf{1}-\mathsf{P}(V))V, V)    
\end{equation*}
for some measurable function $m_{0}$. Moving forward, we define the density $L$ pointwise by
\begin{equation*}
L(X, V):=\left(\prod_{i<j}((v_{i}-v_{j})^{2}+(x_{i}-x_{j})^{2})\right)^{-\frac{N-2}{N(N-1)}}.    
\end{equation*}
We shall use this result in Section \ref{furthersym} below when we come to derive a PDE that the scattering map $\sigma$ of a billiard flow $\{T^{t}_{\sigma}\}_{t\in\mathbb{R}}$ satisfies if it is to admit $\lambda:=L\mathscr{H}$ (and, in turn, $\mu$) as an invariant measure. 
\section{Analysis over the Region $\mathcal{M}_{t}^{+}$}\label{emmplus}
In this Section, we consider the properties of the billiard maps $T^{t}_{\sigma}$ over the region $\mathcal{M}_{t}^{+}$ in order to derive a Jacobian equation that any smooth scattering map $\sigma$ satisfies on the interior of $\mathcal{V}^{-}$. In addition, we derive an additional symmetry condition on the density $M$ which guarantees that the associated measure $\mu$ is invariant under the flow $\{T^{t}_{\sigma}\}_{t\in\mathbb{R}}$ generated by $\sigma$. We consider only the case that $t>0$ in what follows, as we do not seek to derive the second boundary-value problem that the inverse $N$-body scattering map $\sigma^{-1}$ satisfies owing to assumption \ref{Ass} above.
\subsection{Further Auxiliary Maps and Identities}\label{furtheraux}
Suppose now that $\Phi\in C^{0}_{c}(\mathbb{R}^{2N}, \mathbb{R})$ is such that $\mathscr{H}(\mathrm{supp}(\Phi)\cap\mathcal{M}_{t}^{+})>0$. As was the case when considering the integral \eqref{inttofind} above, when considering the new integral given by
\begin{equation}\label{intofinterestpos}
\int_{T^{-t}_{\sigma}\circ\Psi(\mathcal{U}_{t}^{+})}(\Phi\circ T^{t}_{\sigma}\circ\Psi)(M\circ \Psi)\omega\,d\mathscr{L}, 
\end{equation}
elimination of the composition $T^{t}_{\sigma}\circ\Psi$ from the argument of $\Phi$ is once again complicated by the fact that the dimension of the range of this composition is greater than its range. As such, in analogy with Section \ref{emmteeminus} above we also define the \emph{reduced flow map} $\widetilde{T}^{t}_{\sigma}$ pointwise by
\begin{equation*}
\begin{array}{c}
\displaystyle \widetilde{T}^{t}_{\sigma}(\zeta):=\sum_{i=1}^{N}e^{N+2}_{i}\otimes e^{N+2}_{i}\zeta+\tau(\zeta)e^{N+2}_{i}\otimes e^{N}_{i}\psi(\zeta) +(t-\tau(\zeta))e^{N+2}_{i}\otimes e^{N}_{i}\sigma(\psi(\zeta)) \vspace{2mm} \\
\displaystyle +\sum_{j=1}^{2}e^{N+2}_{N+j}\otimes e^{N}_{j}\sigma(\psi(\zeta))
\end{array}
\end{equation*}
for $\zeta\in \mathcal{U}_{t}^{+}$. We note the fact that $\widetilde{T}_{\sigma}^{t}$ is a diffeomorphism with inverse map given pointwise by
\begin{equation*}
(\widetilde{T}_{\sigma}^{t})^{-1}(\zeta)=\widetilde{T}_{\sigma}^{-t}(\zeta).    
\end{equation*}
In this case, the reduced flow map is more complicated in structure and it depends explicitly on the scattering map $\sigma$. Unlike the previous section, we also require a \emph{residual flow map} $\overline{T}_{\sigma}$ given pointwise by
\begin{equation*}
\overline{T}_{\sigma}(\zeta):=\sum_{j=3}^{N}e^{2N}_{N+j}\otimes e_{j}^{N}\sigma(\psi(\zeta))    
\end{equation*}
for $\zeta\in\mathcal{U}_{t}^{+}$. For the purposes of applying the Change of Variables Theorem in the integral \eqref{intofinterestpos}, we require the following result on the structure of the Jacobian of this reduced flow map $\widetilde{T}^{t}_{\sigma}$.
\begin{lem}\label{sigmadet}
Suppose $N\geq 3$ and $t>0$. It holds that
\begin{equation*}
\begin{array}{c}
\displaystyle \mathrm{det}(D\widetilde{T}^{t}_{\sigma}(\zeta))= \vspace{2mm} \\
\displaystyle \frac{\sigma_{1}(\psi(\zeta))-\sigma_{2}(\psi(\zeta))}{u_{1}-u_{2}}\left(\frac{x_{1}+tu_{1}-x_{2}-tu_{2}}{x_{1}-x_{2}}\right)^{N-2}\mathrm{det}(D\sigma(\psi(\zeta)))    
\end{array}
\end{equation*}
for all $\zeta\in(\mathcal{U}_{t}^{+})^{\circ}$.
\end{lem}
\begin{proof}
It can be shown that the derivative matrix $D\widetilde{T}^{t}_{\sigma}:(\mathcal{U}_{t}^{+})^{\circ}\rightarrow\mathbb{R}^{(N+2)\times (N+2)}$ is given block-wise by
\begin{equation*}
D\widetilde{T}^{t}_{\sigma}(\zeta)=
\left(
\begin{array}{cc}
\begin{array}{cc}
A(\zeta, t) & B(\zeta, t)
\end{array} & C(\zeta, t) \vspace{2mm}\\
D(\zeta, t) & E(\zeta) \vspace{2mm}\\
\begin{array}{cc}
F(\zeta, t) & G(\zeta, t)
\end{array} & H(\zeta, t)
\end{array}
\right)^{T}\in\mathbb{R}^{(N+2)\times (N+2)}
\end{equation*}
for $\zeta\in(\mathcal{U}_{t}^{+})^{\circ}$ and $t\in\mathbb{R}$, where the matrix-valued map $A(\cdot, t):(\mathcal{U}_{t}^{+})^{\circ}\rightarrow\mathbb{R}^{2\times 2}$ is defined pointwise by
\begin{equation*}
\begin{array}{c}
\displaystyle A_{i, j}(\zeta, t):= \vspace{2mm}\\
\displaystyle \delta_{i,j}+(-1)^{i}\left(\frac{u_{j}-\sigma_{j}(\psi(\zeta))}{u_{1}-u_{2}}+\frac{t-\tau(\zeta)}{\tau(\zeta)}\sum_{k=3}^{N}(x_{1}-x_{2})^{-1}\left(\sum_{\ell=1}^{2}\widetilde{\delta}_{i, \ell}x_{\ell}-x_{k}\right)\frac{\partial\sigma_{j}}{\partial v_{k}}(\psi(\zeta))\right),
\end{array}
\end{equation*}
the map $B(\cdot, t):(\mathcal{U}_{t}^{+})^{\circ}\rightarrow\mathbb{R}^{2\times (N-2)}$ is given by
\begin{equation*}
\begin{array}{c}
\displaystyle B_{i,j}(\zeta, t):= \vspace{2mm}\\
\displaystyle (-1)^{i}\left((u_{1}-u_{2})^{-1}\sum_{\ell=1}^{2}\widetilde{\delta}_{i, \ell}u_{\ell}-\sigma_{j+2}+\frac{t-\tau(\zeta)}{\tau(\zeta)}\sum_{k=3}^{N}(x_{1}-x_{2})^{-1}\left(\sum_{\ell=1}^{2}\widetilde{\delta}_{i\ell}x_{\ell}-x_{k}\right)\frac{\partial\sigma_{j+2}}{\partial v_{k}}(\psi(\zeta))\right),
\end{array}
\end{equation*}
the map $C:(\mathcal{U}_{t}^{+})^{\circ}\rightarrow\mathbb{R}^{2\times 2}$ is given by
\begin{equation*}
C_{i,j}(\zeta):=(-1)^{i}\frac{1}{\tau(\zeta)}\sum_{k=3}^{N}(x_{1}-x_{2})^{-1}\left(\sum_{\ell=1}^{2}\widetilde{\delta}_{i\ell}x_{\ell}-x_{k}\right)\frac{\partial\sigma_{j}}{\partial v_{k}},
\end{equation*}
the map $D(\cdot, t):(\mathcal{U}_{t}^{+})^{\circ}\rightarrow\mathbb{R}^{(N-2)\times N}$ is given by
\begin{equation*}
D_{i,j}(\zeta, t):=\frac{x_{1}+tu_{1}-x_{2}-tu_{2}}{x_{1}-x_{2}}\frac{\partial\sigma_{j}}{\partial v_{i+2}}(\psi(\zeta))
\end{equation*}
the map $E:(\mathcal{U}_{t}^{+})^{\circ}\rightarrow\mathbb{R}^{(N-2)\times 2}$ is given by
\begin{equation*}
E_{i,j}(\zeta):=-\frac{1}{\tau(\zeta)}\frac{\partial\sigma_{j}}{\partial v_{i+2}}(\psi(\zeta)),
\end{equation*}
the map $F(\cdot, t):(\mathcal{U}_{t}^{+})^{\circ}\rightarrow\mathbb{R}^{2\times 2}$ is given by
\begin{equation*}
\begin{array}{c}
\displaystyle F_{i,j}(\zeta, t):=\vspace{2mm}\\
\displaystyle \tau(\zeta)\delta_{i, j}+(-1)^{i}\left(\frac{1}{\tau(\zeta)}\frac{u_{j}-\sigma_{j}}{u_{1}-u_{2}}+(t-\tau(\zeta))\sum_{k=1}^{N}(x_{1}-x_{2})^{-1}\left(\sum_{\ell=1}^{2}\widetilde{\delta}_{i, \ell}x_{\ell}-x_{k}\right)\frac{\partial\sigma_{j}}{\partial v_{k}}(\psi(\zeta))\right),
\end{array}
\end{equation*}
the map $G(\cdot, t):(\mathcal{U}_{t}^{+})^{\circ}\rightarrow\mathbb{R}^{2\times (N-2)}$ is given by
\begin{equation*}
\begin{array}{c}
\displaystyle G_{i,j}(\zeta, t):= \vspace{2mm}\\
\displaystyle (-1)^{i+1}(u_{1}-u_{2})^{-1}\left(\sum_{\ell=1}^{2}\widetilde{\delta}_{i, \ell}x_{\ell}-x_{j}-\tau(\zeta)\psi_{j}(\zeta)\right)\vspace{2mm} \\
\displaystyle+(-1)^{i}\left(\frac{x_{1}-x_{2}}{(u_{1}-u_{2})^{2}}\sigma_{j}+(t-\tau(\zeta))\sum_{k=1}^{N}(x_{1}-x_{2})^{-1}\left(\sum_{\ell=1}^{2}\widetilde{\delta}_{i, \ell}x_{\ell}-x_{k}\right)\frac{\partial\sigma_{j}}{\partial v_{k}}\right)
\end{array}
\end{equation*}
and finally the map $H(\cdot, t):(\mathcal{U}_{t}^{+})^{\circ}\rightarrow\mathbb{R}^{2\times 2}$ is given by
\begin{equation*}
H_{i,j}(\zeta, t):=(-1)^{i}\sum_{k=1}^{N}(x_{1}-x_{2})^{-1}\left(\sum_{\ell=1}^{2}\widetilde{\delta}_{i, \ell}x_{\ell}-x_{k}\right)\frac{\partial\sigma_{j}}{\partial v_{k}}.
\end{equation*}
By applying elementary row and column operations on the matrix $D\widetilde{T}^{t}_{\sigma}(\zeta)^{T}$, it holds that
\begin{equation}\label{transposematrix}
\begin{array}{c}
\displaystyle \prod_{i=1}^{2}\mathsf{E}\left(i, N+i, \frac{u_{1}-u_{2}}{x_{1}-x_{2}}\right)D\widetilde{T}^{t}_{\sigma}(\zeta)^{T}\prod_{i=1}^{2}\mathsf{E}\left(i, N+i, -\frac{u_{1}-u_{2}}{x_{1}+tu_{1}-x_{2}-tu_{2}}\right)\times\vspace{2mm}\\
\displaystyle\mathsf{E}\left(N+1, N+2, \frac{u_{2}-\sigma_{2}}{u_{2}-\sigma_{1}}\right)
\end{array}
\end{equation}
equals the block matrix
\begin{equation*}
\left(
\begin{array}{cc}
\displaystyle \frac{x_{1}+tu_{1}-x_{2}-tu_{2}}{x_{1}-x_{2}}D\sigma(\psi(\zeta)) & 0_{2\times N} \vspace{2mm}\\
K(\zeta, t) & L(\zeta, t)
\end{array}
\right),
\end{equation*}
where $K(\cdot, t):\mathcal{U}_{t}^{+}\rightarrow\mathbb{R}^{2\times N}$ is a map whose values will be immaterial to the value of the determinant, and $L(\cdot, t):\mathcal{U}_{t}^{+}\rightarrow\mathbb{R}^{2\times 2}$ is given by
\begin{equation*}
L(\zeta, t):=
\left(
\begin{array}{cc}
\displaystyle -\frac{x_{1}-x_{2}}{x_{1}+tu_{1}-x_{2}-tu_{2}}\frac{u_{2}-\sigma_{1}}{u_{1}-u_{2}} & 0 \vspace{2mm}\\
\star & \displaystyle -\frac{x_{1}-x_{2}}{x_{1}+tu_{1}-x_{2}-tu_{2}}\frac{\sigma_{1}-\sigma_{2}}{u_{2}-\sigma_{1}}
\end{array}
\right).
\end{equation*}
By expanding the determinant of the matrix \eqref{transposematrix} down the final column, which is made up entirely of zeros aside from the bottom right-most entry, and performing the very same action yet once more on the resultant $(N+1)\times (N+1)$ minor, we find that 
\begin{equation*}
\mathrm{det}(D\widetilde{T}^{\sigma}(\zeta)^{T})=\frac{\sigma_{1}-\sigma_{2}}{u_{1}-u_{2}}\left(\frac{x_{1}+tu_{1}-x_{2}-tu_{2}}{x_{1}-x_{2}}\right)^{N-2}\mathrm{det}(D\sigma(\psi(\zeta))),        
\end{equation*}
whence the claimed formula in the statement of the Proposition holds true.
\end{proof}
Analogously to lemma \ref{densitycomp1}, we also require the following structural result which we state without proof.
\begin{lem}\label{anotheromegaid}
Suppose $N\geq 3$. It holds that
\begin{equation*}
\begin{array}{c}
\omega(\zeta)= \vspace{2mm} \\
\displaystyle \left(\frac{(u_{1}-u_{2})^{2}+(x_{1}-x_{2})^{2}}{(u_{1}-u_{2})^{2}+(x_{1}+tu_{1}-x_{2}-tu_{2})^{2}}\right)^{\frac{N-2}{2}}\left|\frac{x_{1}+tu_{1}-x_{2}-tu_{2}}{x_{1}-x_{2}}\right|^{N-2}\times \vspace{2mm} \\
\displaystyle \left|\frac{\sigma_{1}-\sigma_{2}}{x_{1}-x_{2}}\right|\left(\sum_{i>j}(\sigma_{i}-\sigma_{j})^{2}\right)^{-\frac{1}{2}}\left(\sum_{i>j}(x_{i}-x_{j})^{2}\right)^{-\frac{1}{2}}\omega(\widetilde{T}^{t}_{\sigma}(\zeta))
\end{array}
\end{equation*}
for $\zeta\in(\mathcal{U}_{t}^{+})^{\circ}$ and $t>0$.
\end{lem}
We also require the following result containing some rational identities the components of the parametrisation map $\psi$ satisfy. Its straightforward algebraic demonstration is once again left to the reader. 
\begin{lem}
Suppose $N\geq 3$. It holds that
\begin{equation}\label{ratioidentitysimple}
\frac{x_{i}-x_{j}}{x_{1}-x_{2}}=\frac{\psi_{i}(X, U)-\psi_{j}(X, U)}{u_{1}-u_{2}}    
\end{equation}
for $\zeta\in (\mathcal{U}_{t}^{+})^{\circ}$ and $i, j\in\{1, ..., N\}$ with $i<j$.
\end{lem}
\subsection{Derivation of the Jacobian PDE}
Making use of the identities of Section \ref{furtheraux} above, we now derive PDE for the associated second boundary-value problem which any regular scattering map $\sigma$ satisfies. For any $\Phi\in C^{0}_{c}(\mathbb{R}^{2N}, \mathbb{R})$ with the property that $\mathscr{H}(\mathrm{supp}(\Phi)\cap\mathcal{M}_{t}^{+})>0$, we find that 
\begin{equation}\label{deductions}
\begin{array}{cl}
& \langle T^{t}_{\sigma}\#\mu, \Phi\rangle \vspace{2mm} \\
= & \displaystyle\int_{\widetilde{T}^{-t}_{\sigma}(\mathcal{U}_{t}^{+})}\Phi(T^{t}_{\sigma}(\Psi(\zeta)))M(\Psi(\zeta))\omega(\zeta)\,d\mathscr{L}(\zeta) \vspace{2mm} \\
= & \displaystyle\int_{\widetilde{T}^{-t}_{\sigma}(\mathcal{U}_{t}^{+})}\Phi(\widetilde{T}^{t}_{\sigma}\times\overline{T}_{\sigma}(\zeta))M(\Psi(\zeta))\omega(\zeta)\,d\mathscr{L}(\zeta) \vspace{2mm}\\   
 = & \displaystyle \int_{\mathcal{U}_{t}^{+}}\Phi(\widetilde{\Psi}(\zeta), \overline{T}_{\sigma}(\widetilde{T}^{-t}_{\sigma}(\zeta)))M(\widetilde{T}^{-t}_{\sigma}(\zeta), \widetilde{\psi}(\widetilde{T}^{-t}_{\sigma}(\zeta)))\omega(\widetilde{T}^{-t}_{\sigma}(\zeta))\left|\mathrm{det}(D\widetilde{T}^{t}_{\sigma}(\widetilde{T}^{-t}_{\sigma}(\zeta)))^{-1}\right|\,d\mathscr{L}(\zeta)\vspace{2mm}\\
 = & \displaystyle \int_{\mathcal{U}_{t}^{+}}\Phi(\Psi(\zeta))M(\widetilde{T}^{-t}_{\sigma}(\zeta), \widetilde{\psi}(\widetilde{T}^{-t}_{\sigma}(\zeta)))\omega(\widetilde{T}^{-t}_{\sigma}(\zeta))\left|\mathrm{det}(D\widetilde{T}^{t}_{\sigma}(\widetilde{T}^{-t}_{\sigma}(\zeta)))^{-1}\right|\,d\mathscr{L}(\zeta).
\end{array}
\end{equation}
Thus, we find from \eqref{deductions} that
\begin{equation*}
\langle \widetilde{T}_{\sigma}^{t}\#\mu, \Phi\rangle=\langle\mu, \Phi\rangle    
\end{equation*}
for all $\Phi\in C^{0}_{c}(\mathbb{R}^{N}, \mathbb{R})$ satisfying the condition $\mathscr{H}(\mathrm{supp}(\Phi)\cap\mathcal{M}_{t}^{+})>0$ if 
\begin{equation*}
\pm M(\Psi(\zeta))\omega(\zeta)\mathrm{det}(D\widetilde{T}_{\sigma}^{t}(\zeta))^{-1}=M(\Psi(\widetilde{T}_{\sigma}^{t}(\zeta)))\omega(\widetilde{T}_{\sigma}^{t}(\zeta))
\end{equation*}
for all $\zeta\in\mathcal{U}_{t}^{+}$, whence by Lemma \ref{sigmadet} if the $N$-body scattering map $\sigma$ satisfies
\begin{equation}\label{continuethis}
\begin{array}{c}
\displaystyle \pm\frac{\sigma_{1}(\psi(\zeta))-\sigma_{2}(\psi(\zeta))}{u_{1}-u_{2}}\left(\frac{x_{1}+tu_{1}-x_{2}-tu_{2}}{x_{1}-x_{2}}\right)^{N-2}\mathrm{det}(D\sigma(\psi(\zeta))) \vspace{2mm} \\
=\displaystyle \frac{M(\Psi(\zeta))\omega(\zeta)}{M(T^{t}_{\sigma}(\Psi(\zeta)))\omega(\widetilde{T}^{t}_{\sigma}(\zeta))}.
\end{array}
\end{equation}
At this point, we derive sufficient conditions on the density $M$ so that $\mu=M\lambda$ is an invariant measure of the billiard flow $\{T^{t}_{\sigma}\}_{t\in\mathbb{R}}$.
\subsection{Further Symmetry of the Density $M$}\label{furthersym}
It has already been demonstrated in Section \ref{emmteeminus} above that if the density $M$ of the measure $\mu=M\mathscr{H}$ is of the shape
\begin{equation*}
M(Z)=\left(\prod_{i<j}(v_{i}-v_{j})^{2}+(x_{i}-x_{j})^{2}\right)^{-\frac{N-2}{N(N-1)}}m_{0}(X\cdot (\mathsf{E}(V)\mathbf{1}-\mathsf{P}(V))V, V)        
\end{equation*}
for some measurable function $m_{0}$, then for each fixed $t\in\mathbb{R}$ it holds that $\langle T^{t}_{\sigma}\#\mu, \Phi\rangle=\langle \mu, \Phi\rangle$ for all $\Phi\in C^{0}_{c}(\mathbb{R}^{2N}, \mathbb{R})$ with the necessary support property. It will be helpful to `desymmetrise' the density $L$ of the Liouville measure $\lambda$ on $\mathcal{M}$ for our purposes in what follows. In particular, we shall use the fact (as found above) that
\begin{equation*}
\left(\frac{(u_{1}-u_{2})^{2}+(x_{1}+tu_{1}-x_{1}-tu_{2})^{2}}{(u_{1}-u_{2})^{2}+(x_{1}-x_{2})^{2}}\right)^{\frac{N-2}{2}}M(X+t\psi(\zeta), \psi(\zeta))=M(X, \psi(\zeta))     
\end{equation*}
for $Z\in\mathcal{M}^{\circ}$. As a consequence, it follows that the ratio of densities appearing on the right-hand side of identity \eqref{continuethis} is given by
\begin{equation}\label{ratioidentity}
\begin{array}{cl}
& \displaystyle\frac{M(Z)}{M(T^{t}_{\sigma}(Z))} \vspace{2mm} \\
= & \displaystyle \left((u_{1}-u_{2})^{2}+(x_{2}-x_{2})^{2}\right)^{-\frac{N-2}{2}}\times \vspace{2mm} \\
& \displaystyle\left((\sigma_{1}-\sigma_{2})^{2}+(x_{1}+\tau v_{1}+(t-\tau)\sigma_{1}-x_{2}-\tau v_{2}-(t-\tau)\sigma_{2})^{2}\right)^{\frac{N-2}{2}}\times \vspace{2mm}\\
& \displaystyle\frac{m_{0}(X\cdot(\mathsf{E}(V)\mathbf{1}-\mathsf{P}(V)V, V))}{m_{0}((X+\tau V+(t-\tau)\sigma)\cdot(\mathsf{E}(\sigma)\mathbf{1}-\mathsf{P}(\sigma)\sigma), \sigma)}
\end{array}
\end{equation}
for all $Z\in\mathcal{M}^{\circ}$. We note that since
\begin{equation*}
\sigma\cdot(\mathsf{E}(\sigma)\mathbf{1}-\mathsf{P}(\sigma)\sigma)=0,
\end{equation*}
it follows that
\begin{equation*}
\frac{m_{0}(X\cdot(\mathsf{E}(V)\mathbf{1}-\mathsf{P}(V)V, V))}{m_{0}((X+\tau V+(t-\tau)\sigma)\cdot(\mathsf{E}(\sigma)\mathbf{1}-\mathsf{P}(\sigma)\sigma), \sigma)}=\frac{m(X, V)}{m(X+\tau V, \sigma)}.
\end{equation*}
Thus, in order that the PDE be independent of the density $M$, we stipulate additionally that $m$ admit the symmetry
\begin{equation}
m(X, \sigma(V))=m(X, V)    
\end{equation}
for all $Z=(X, V)\in \mathcal{M}^{\circ}$. Thus, it follows that
\begin{equation*}
\frac{m(X, V)}{m(X+\tau V, \sigma)}=\frac{m(X, V)}{m(X+\tau V, V)}=1.
\end{equation*}
Finally, as it holds that
\begin{equation*}
\begin{array}{c}
\displaystyle\left(\frac{(\sigma_{1}-\sigma_{2})^{2}+(x_{1}+\tau v_{1}+(t-\tau)\sigma_{1}-x_{2}-\tau v_{2}-(t-\tau)\sigma_{2})^{2}}{(u_{1}-u_{2})^{2}+(x_{1}-x_{2})^{2}}\right)^{\frac{N-2}{2}} \vspace{2mm} \\
\displaystyle = \left(\frac{(u_{1}-u_{2})^{2}+(x_{1}+tu_{1}-x_{2}-tu_{2})^{2}}{(u_{1}-u_{2})^{2}+(x_{1}-x_{2})^{2}}\right)^{\frac{N-2}{2}}\left|\frac{\sigma_{1}-\sigma_{2}}{u_{1}-u_{2}}\right|^{N-2},
\end{array}
\end{equation*}
we deduce from \eqref{ratioidentity} that $\sigma$ satisfies
\begin{equation*}
\begin{array}{c}
\displaystyle \pm\frac{\sigma_{1}(\psi(\zeta))-\sigma_{2}(\psi(\zeta))}{u_{1}-u_{2}}\left|\frac{x_{1}+tu_{1}-x_{2}-tu_{2}}{x_{1}-x_{2}}\right|^{N-2}\mathrm{det}(D\sigma(\psi(\zeta))) \vspace{2mm} \\
=\displaystyle \frac{M(\Psi(\zeta))\omega(\zeta)}{M(T^{t}_{\sigma}(\Psi(\zeta)))\omega(\widetilde{T}^{t}_{\sigma}(\zeta))}
\end{array}
\end{equation*}
if and only if
\begin{equation}\label{nextreason}
\begin{array}{c}
\displaystyle\mathrm{det}(D\sigma(\psi(\zeta)))=\pm\left|\frac{\sigma_{1}-\sigma_{2}}{u_{1}-u_{2}}\right|^{N-2}\left|\frac{u_{1}-u_{2}}{x_{1}-x_{2}}\right|\left(\sum_{i<j}(\sigma_{i}-\sigma_{j})^{2}\right)^{-\frac{1}{2}}\left(\sum_{i<j}(x_{i}-x_{j})^{2}\right)^{\frac{1}{2}}.
\end{array}
\end{equation}
We note that
\begin{equation*}
\begin{array}{cl}
& \displaystyle \left|\frac{v_{1}-v_{2}}{x_{1}-x_{2}}\right|\left(\sum_{i<j}(x_{i}-x_{j})^{2}\right)^{\frac{1}{2}} \vspace{2mm} \\
= & \displaystyle \left((v_{1}-v_{2})^{2}\sum_{i<j}\left(\frac{x_{i}-x_{j}}{x_{1}-x_{2}}\right)^{2}\right)^{\frac{1}{2}}\vspace{2mm} \\
\overset{\eqref{ratioidentitysimple}}{=} & \displaystyle \left((v_{1}-v_{2})^{2}\sum_{i<j}\left(\frac{v_{i}-v_{j}}{v_{1}-v_{2}}\right)^{2}\right)^{\frac{1}{2}}\vspace{2mm} \\
= & \displaystyle \left(\sum_{i<j}(v_{i}-v_{j})^{2}\right)^{\frac{1}{2}}
\end{array}
\end{equation*}
for all $Z=(X, V)\in\mathcal{M}^{\circ}$, whence from \eqref{nextreason} above we infer that $\sigma$ satisfies the PDE
\begin{equation*}
\mathrm{det}(D\sigma(V))=\pm\left|\frac{\sigma_{1}(V)-\sigma_{2}(V)}{v_{1}-v_{2}}\right|^{N-2}\left(\sum_{i<j}(\sigma_{i}(V)-\sigma_{j}(V))^{2}\right)^{-\frac{1}{2}}\left(\sum_{i<j}(v_{i}-v_{j})^{2}\right)^{\frac{1}{2}}    
\end{equation*}
for all $V\in\mathcal{V}^{\circ}$. Now, just as we observed in Section \ref{emmteeminus} above, there is once again a spurious bias on quantities indexed by 1 and 2. This is again a consequence of our choice of chart map $\Psi$. Indeed, by instead slaving the velocity variables $v_{1}, ..., v_{k-1}, v_{k+1}, ..., v_{\ell-1}, v_{\ell+1}, ..., v_{N}$ to $v_{k}$ and $v_{\ell}$ and redefining the chart map $\Psi$ appropriately (note that $k=1$ and $\ell=2$ in the above), in turn repeating the above arguments it can be shown that $\sigma$ also satisfies one of the equations
\begin{equation*}
\mathrm{det}(D\sigma(V))=\pm\left|\frac{\sigma_{k}(V)-\sigma_{\ell}(V)}{v_{k}-v_{\ell}}\right|^{N-2}\left(\sum_{i<j}(\sigma_{i}(V)-\sigma_{j}(V))^{2}\right)^{-\frac{1}{2}}\left(\sum_{i<j}(v_{i}-v_{j})^{2}\right)^{\frac{1}{2}}    
\end{equation*}
for any $k, \ell\in\{1, ..., N\}$ with $k<\ell$, whence by symmetrisation we conclude that if the scattering map $\sigma$ is a classical solution of one of the Jacobian PDE
\begin{equation*}
\begin{array}{c}
\displaystyle \mathrm{det}(D\sigma(V))= \vspace{2mm} \\
\displaystyle\pm\left(\prod_{i<j}\left(\frac{\sigma_{i}(V)-\sigma_{j}(V)}{v_{i}-v_{j}}\right)^{2}\right)^{\frac{N-2}{N(N-1)}}\left(\sum_{i<j}(\sigma_{i}(V)-\sigma_{j}(V))^{2}\right)^{-\frac{1}{2}}\left(\sum_{i<j}(v_{i}-v_{j})^{2}\right)^{\frac{1}{2}}    
\end{array}
\end{equation*}
for $V\in\mathcal{V}^{\circ}$, then it generates an invariant measure. 
\subsection{Proofs of Theorems \ref{firstmainresult} and \ref{densitythm}}
Let us now summarise the observations of Sections \ref{emmteeminus} and \eqref{emmplus} in the form of the proof of our main Theorem \ref{firstmainresult}, which we state one again for the convenience of the reader.
\begin{thm}
Let $N\geq 3$. Suppose $\sigma\in C^{0}(\mathcal{V}^{-}, \mathcal{V}^{+})\cap C^{1}((\mathcal{V}^{-})^{\circ}, (\mathcal{V}^{+})^{\circ})$ is a given $N$-body scattering map. It holds that $\sigma$ generates a billiard flow $\{T^{t}_{\sigma}\}_{t\in\mathbb{R}}$ on $\mathcal{M}$ with the property 
\begin{equation*}
T^{t}_{\sigma}\#\lambda=\lambda    
\end{equation*}
for all $t\in\mathbb{R}$ if and only if $\sigma$ is a classical solution of one the second boundary-value problems:
\begin{equation*}
\left\{
\begin{array}{l}
\displaystyle \mathrm{det}(D\sigma(V))=\pm\frac{F(V)}{F(\sigma(V))} \quad \text{for}\hspace{2mm}V\in(\mathcal{V}^{-})^{\circ}, \vspace{2mm}\\
\displaystyle \sigma(\mathcal{V}^{-})=\mathcal{V}^{+},
\end{array}
\right.
\end{equation*}
where $F$ is defined in \eqref{RHS} above.
\end{thm}
\begin{proof}
We consider the proof of sufficiency to begin. Let $E\subset\mathcal{M}$ be a bounded measurable set, and let $\{\Phi_{j}\}_{j=1}^{\infty}\subset C^{0}_{c}(\mathbb{R}^{2N}, \mathbb{R})$ denote any sequence with the property that
\begin{equation*}
\lambda(E)=\lim_{j\rightarrow\infty}\int_{\mathcal{M}}\Phi_{j}\,d\lambda.    
\end{equation*}
Let $t\in\mathbb{R}$ be given and fixed, and let $\sigma$ be an $N$-body scattering map of class $C^{0}(\mathcal{V}^{-}, \mathcal{V}^{+})\cap C^{1}((\mathcal{V}^{-})^{\circ}, (\mathcal{V}^{+})^{\circ})$. We have that
\begin{equation*}
\begin{array}{lcl}
(T^{t}_{\sigma}\#\lambda)(E) & = & \displaystyle \lim_{j\rightarrow\infty}\displaystyle \langle T^{t}_{\sigma}\#\lambda, \Phi_{j}\rangle \vspace{2mm} \\
& = & \displaystyle \lim_{j\rightarrow\infty}\int_{\mathcal{M}^{\circ}}\Phi_{j}(T^{t}_{\sigma}(Z))\,d\lambda(Z) \vspace{2mm}\\
& = & \displaystyle \lim_{j\rightarrow\infty}\underbrace{\int_{T^{-t}_{\sigma}(\mathcal{M}_{t}^{-})}\Phi_{j}(T^{t}_{\sigma}(Z))\,d\lambda(Z)}_{I_{t}(\Phi_{j}):=}+\lim_{j\rightarrow\infty}\underbrace{\int_{T^{-t}_{\sigma}(\mathcal{M}_{t}^{+})}\Phi(T^{t}_{\sigma}(Z))\,d\lambda(Z)}_{J_{t}(\Phi_{j}):=}
\end{array}
\end{equation*}
by Assumption \ref{Ass}. Let us focus on the first of these integrals. Now,
\begin{equation*}
\begin{array}{lcl}
I_{t}(\Phi_{j}) & = & \displaystyle\int_{T_{\sigma}^{-t}(\mathcal{M}_{t}^{-})}\Phi_{j}(T^{t}_{\sigma}(Z))\,d\lambda(Z) \vspace{2mm}\\
& = & \displaystyle \int_{\widetilde{T}^{-t}(\mathcal{U}_{t}^{-})}\Phi_{j}(\widetilde{T}^{t}(\zeta), \overline{\psi}(\zeta))L(\zeta, \overline{\psi}(\zeta))\omega(\zeta)d\mathscr{L}(\zeta)\vspace{2mm}\\
& = & \displaystyle \int_{\mathcal{U}_{t}^{-}}\Phi_{j}(\Psi(\zeta))L(\widetilde{T}^{-t}(\zeta), \overline{\psi}(\zeta))\omega(\widetilde{T}^{-t}(\zeta))\left|\frac{x_{1}-tu_{1}-x_{2}+tu_{2}}{x_{1}-x_{2}}\right|^{N-2}\,d\mathscr{L}(\zeta),
\end{array}
\end{equation*}
and so by Lemma \ref{densitycomp1} it holds that
\begin{equation*}
\begin{array}{lcl}
I_{t}(\Phi_{j}) & = & \displaystyle \int_{\mathcal{U}_{t}^{-}}\Phi_{j}(\Psi(\zeta))L(\Psi(\zeta))\omega(\zeta)\,d\mathscr{L}(\zeta)\vspace{2mm}\\
& = & \displaystyle \int_{\Psi(\mathcal{U}_{t}^{-})}\Phi_{j}(Z)L(Z)\,d\mathscr{H}(Z)\vspace{2mm} \\
& = & \displaystyle \int_{\mathcal{M}_{t}^{-}}\Phi_{j}\,d\lambda.
\end{array}
\end{equation*}
Let us now consider the second integral $J_{t}(\Phi_{j})$. Indeed, by way of similar calculations, it holds that
\begin{equation*}
J_{t}(\Phi_{j}) = \int_{\mathcal{U}_{t}^{+}}\Phi_{j}(\Psi(\zeta))L(\Psi(\widetilde{T}^{-t}_{\sigma}(\zeta)))\omega(\widetilde{T}^{-t}_{\sigma}(\zeta))\left|\mathrm{det}(D\widetilde{T}^{t}_{\sigma}(\widetilde{T}^{-t}_{\sigma}(\zeta)))\right|^{-1}d\mathscr{L}(\zeta). 
\end{equation*}
By Lemmas \ref{sigmadet} and \ref{anotheromegaid}, it holds that 
\begin{equation*}
\mathrm{det}(D\widetilde{T}^{t}_{\sigma}(\widetilde{T}^{-t}_{\sigma}(\zeta)))^{-1}=\frac{L(\Psi(\zeta))\omega(\zeta)}{L(\Psi(\widetilde{T}^{-t}_{\sigma}(\zeta)))\omega(\widetilde{T}^{-t}_{\sigma}(\zeta))} 
\end{equation*}
for all $\zeta\in\mathcal{U}_{t}^{+}$ if and only if
\begin{equation}\label{showimequiv}
\mathrm{det}(D\widetilde{T}^{t}_{\sigma}(\zeta))^{-1}=\frac{L(\Psi(\widetilde{T}^{t}_{\sigma}(\zeta)))\omega(\widetilde{T}^{t}_{\sigma}(\zeta))}{L(\Psi(\zeta))\omega(\zeta)}     
\end{equation}
for all $\zeta\in\widetilde{T}^{-t}_{\sigma}(\mathcal{U}_{t}^{+})$. By the calculations in Section \ref{emmplus} above, identity \eqref{showimequiv} holds true if and only if
\begin{equation*}
\mathrm{det}(D\sigma(\psi(\zeta)))=\pm\left|\frac{\sigma_{1}-\sigma_{2}}{u_{1}-u_{2}}\right|^{N-2}\left|\frac{u_{1}-u_{2}}{x_{1}-x_{2}}\right|\left(\sum_{i<j}(\sigma_{i}-\sigma_{j})^{2}\right)^{-\frac{1}{2}}\left(\sum_{i<j}(x_{i}-x_{j})^{2}\right)^{\frac{1}{2}}.   
\end{equation*}
Thus, by the assumption that $\sigma$ is a classical solution of one of the second boundary-value problems in the statement of the Theorem, it holds that
\begin{equation*}
\begin{array}{lcl}
J_{t}(\Phi_{j}) & = & \displaystyle\int_{\mathcal{U}_{t}^{+}}\Phi_{j}(\Psi(\zeta))L(\Psi(\zeta))\omega(\zeta)d\mathscr{L}(\zeta)\vspace{2mm}\\
& = & \displaystyle \int_{\mathcal{M}_{t}^{+}}\Phi_{j}\,d\lambda.
\end{array}
\end{equation*}
As a result, we conclude that
\begin{equation*}
\begin{array}{lcl}
(T^{t}_{\sigma}\#\lambda)(E) & = & \displaystyle \lim_{j\rightarrow\infty}I_{t}(\Phi_{j})+\lim_{j\rightarrow\infty}J_{t}(\Phi_{j})\vspace{2mm} \\
& = & \displaystyle \lim_{j\rightarrow\infty}\int_{\mathcal{M}_{t}^{-}}\Phi_{j}\,d\lambda+\lim_{j\rightarrow\infty}\int_{\mathcal{M}_{t}^{+}}\Phi_{j}\,d\lambda \vspace{2mm} \\
& = & \displaystyle \lim_{j\rightarrow\infty}\int_{\mathcal{M}^{\circ}}\Phi_{j}\,d\lambda \vspace{2mm}\\
& = & \lambda(E),
\end{array}
\end{equation*}
whence $\lambda$ is an invariant measure of the billiard flow $\{T^{t}_{\sigma}\}_{t\in\mathbb{R}}$. The proof of necessity is straightforward and is left to the reader.
\end{proof}
We may also conclude the proof of Theorem \ref{densitythm} with the work established above.
\begin{thm}
Let $N\geq 3$. Suppose $\sigma\in C^{0}(\mathcal{V}^{-}, \mathcal{V}^{+})\cap C^{1}((\mathcal{V}^{-})^{\circ}, (\mathcal{V}^{+})^{\circ})$ is a momentum- and energy-conserving $N$-body scattering map, and let $\{T^{t}_{\sigma}\}_{t\in\mathbb{R}}$ denote the billiard flow on $\mathcal{M}$ it generates. Suppose, moreover, that $\mu\ll \lambda$, with
\begin{equation*}
\frac{d\mu}{d\lambda}=M    
\end{equation*}
for some measurable function $M:\mathcal{M}\rightarrow\mathbb{R}$. If
$M$ is of the form
\begin{equation*}
M(Z)=m_{0}(X\cdot(\mathsf{E}(V)\mathbf{1}-\mathsf{P}(V)V), V)    
\end{equation*}
for $Z=(X, V)\in\mathcal{M}$ and some measurable $m_{0}:\mathbb{R}\times\mathbb{R}^{N}\rightarrow\mathbb{R}$ admitting the scattering symmetry
\begin{equation*}
m_{0}(y, \sigma(V))=m_{0}(y, V)    
\end{equation*}
for $y\in\mathbb{R}$ and $V\in\mathcal{V}^{-}$, then $T^{t}_{\sigma}\#\mu=\mu$ for all $t\in\mathbb{R}$.   
\end{thm}
\begin{proof}
This follows from arguments largely similar to those in the proof of Theorem \ref{firstmainresult}, with the notable difference that the integrals contain an additional contribution to the integrand, namely the density $M$ (as opposed to simply $L$). We leave the details of the proof to the reader.
\end{proof}
With the above established, it is straightforward to generate a large class of invariant measures of a given billiard flow $\{T^{t}_{\sigma}\}_{t\in\mathbb{R}}$ on $\mathcal{M}$. Indeed, suppose that $M:\mathbb{R}^{2N}\rightarrow\mathbb{R}$ is of the form
\begin{equation*}
M(Z)=f(X\cdot(|V|^{2}\mathbf{1}-\mathsf{P}(V)V))g(V),    
\end{equation*}
for any measurable map $f:\mathbb{R}\rightarrow\mathbb{R}$, where $g:\mathbb{R}^{N}\rightarrow\mathbb{R}$ is any measurable map constant on the orbits of points in $\mathbb{R}^{N}$ under the iterated action of $\sigma$. If follows $\mu=M\lambda$ is an invariant measure of the corresponding flow $\{T^{t}_{\sigma}\}_{t\in\mathbb{R}}$. 
\section{Uniqueness of Momentum- and Energy-conserving $N$-body Linear Scattering}\label{uniquelinny}
In this Section, we provide a proof of Theorem \ref{uniquelinear}, namely that there can only be one linear $N$-body scattering map $\sigma$ which conserves both momentum and energy, and that this scattering map generates a billiard flow on the invariant submanifold $\mathcal{M}$ of the tangent bundle of the fundamental table $T\mathcal{Q}$ which admits the Liouville measure $\lambda$ on $\mathcal{M}$ as an invariant measure. Let us begin by demonstrating the proof of the first of these claims.
\begin{prop}
Suppose $N\geq 3$. There is only one linear $N$-body scattering map $\sigma^{\star}$ which conserves both momentum and energy, and it is given pointwise by
\begin{equation*}
\sigma^{\star}(V):=\left(\frac{2}{N}\mathbf{1}\otimes \mathbf{1}-I_{N}\right)V    
\end{equation*}
for $V\in\mathcal{V}^{-}$.
\end{prop}
\begin{proof}
We begin by adopting the ansatz for $\sigma:\mathcal{V}^{-}\rightarrow\mathcal{V}^{+}$ given by
\begin{equation*}
\sigma(V)=AV    
\end{equation*}
for some $A\in\mathrm{GL}(N)$. In addition, as
\begin{equation*}
-\sigma(-\sigma(V))=V    
\end{equation*}
for all $V\in\mathcal{V}^{-}$, it follows that $A^{2}=I_{N}$. Moreover, as it is assumed that $\mathsf{E}(\sigma(V))=\mathsf{E}(V)$ for all $V\in\mathcal{V}^{-}$, it follows that $A\in\mathrm{O}(N)$. With the knowledge that $A\in\mathrm{O}(N)$, we seek a representation of the matrix $A$ of the form
\begin{equation*}
A=\sum_{i=1}^{N}\lambda_{i}(\widehat{a}_{i}\otimes \widehat{a}_{i})    
\end{equation*}
with $\lambda_{i}\in\{-1, 1\}$ for $i\in\{1, ..., N\}$ and vectors $\{\widehat{a}_{i}\}_{i=1}^{N}\subset\mathbb{R}^{N}$ which constitute an orthonomal basis of $\mathbb{R}^{N}$. The conservation of momentum yields an eigenpair of $A^{T}$. Indeed, 
\begin{equation*}
\mathsf{P}(\sigma(V))=\mathsf{P}(V)    
\end{equation*}
if and only if
\begin{equation}\label{linmomeigen}
(A^{T}\mathbf{1}-\mathbf{1})\cdot V=0    
\end{equation}
for all $V\in\mathcal{V}^{-}$. However, as $A^{2}=I_{N}$ it follows that identity \eqref{linmomeigen} holds for \emph{all} $V\in\mathbb{R}^{N}$. By separation, $A^{T}$ admits $\mathbf{1}$ as an eigenvector with corresponding eigenvalue $1$. 

The remaining $N-1$ eigenpairs of $A$ may be inferred from the fact that $\sigma$ maps the polytope $\mathcal{V}^{-}$ onto the polytope $\mathcal{V}^{+}$. Indeed, we recall that
\begin{equation*}
\sigma(V)\cdot E_{i}^{N}\geq 0 \quad \text{whenever}\quad V\cdot E_{i}^{N}\leq 0    
\end{equation*}
where $E_{i}^{N}=e_{i}^{N}-e_{i+1}^{N}$, which holds if and only if 
\begin{equation*}
A^{T}E_{i}^{N}\cdot V\geq 0 \quad \text{whenever}\quad E_{i}^{N}\cdot V\leq 0    
\end{equation*}
for $i\in\{1, ..., N-1\}$. By a contradiction argument, it follows that
\begin{equation*}
A^{T}E_{i}^{N}=\alpha_{i}E_{i}^{N} 
\end{equation*}
for some $\alpha_{i}<0$. However, it must hold that $\alpha_{i}\in\{-1, 1\}$, whence $\alpha_{i}=-1$ for $i\in\{1, ..., N-1\}$. It may be checked that 
\begin{equation}\label{liset}
\{\mathbf{1}, E_{1}^{N}, ..., E_{N-1}^{N}\}    
\end{equation}
constitutes a linearly independent set of vectors in $\mathbb{R}^{N}$. Noting that $A^{2}=I_{N}$ and $A^{T}A=I_{N}$, it follows that $A=A^{T}$. As such, the conservation of linear momentum, kinetic energy, and the range mapping condition $A\mathcal{V}^{-}=\mathcal{V}^{+}$ completely determine the matrix $A$. By applying the Gram-Schmidt algorithm to the linearly independent set of vectors \eqref{liset}, one finds that
\begin{equation*}
A^{\star}:=\frac{2}{N}\mathbf{1}\otimes \mathbf{1}-\sum_{k=1}^{N}e_{k}^{N}\otimes e_{k}^{N},
\end{equation*}
from which follows the claim of the Proposition.
\end{proof}
\begin{rem}
More concretely, the $N$-body scattering matrices $A^{\star}$ are given explictly by
\begin{equation*}
A^{\ast}=\left(
\begin{array}{ccc}
-\frac{1}{3} & \frac{2}{3} & \frac{2}{3} \vspace{2mm}\\
\frac{2}{3} & -\frac{1}{3} & \frac{2}{3} \vspace{2mm}\\
\frac{2}{3} & \frac{2}{3} & -\frac{1}{3}
\end{array}
\right)
\end{equation*}
in the case $N=3$, while by
\begin{equation*}
A^{\star}=\left(
\begin{array}{cccc}
-\frac{1}{2} & \frac{1}{2} & \frac{1}{2} & \frac{1}{2} \vspace{2mm} \\
\frac{1}{2} & -\frac{1}{2} & \frac{1}{2} & \frac{1}{2} \vspace{2mm} \\
\frac{1}{2} & \frac{1}{2} & -\frac{1}{2} & \frac{1}{2} \vspace{2mm} \\
\frac{1}{2} & \frac{1}{2} & \frac{1}{2} & -\frac{1}{2} 
\end{array}
\right)
\end{equation*}
in the case $N=4$, and by
\begin{equation*}
A^{\star}=\left(
\begin{array}{ccccc}
-\frac{3}{5} & \frac{2}{5} & \frac{2}{5} & \frac{2}{5} & \frac{2}{5} \vspace{2mm} \\
\frac{2}{5} & -\frac{3}{5} & \frac{2}{5} & \frac{2}{5} & \frac{2}{5} \vspace{2mm} \\
\frac{2}{5} & \frac{2}{5} & -\frac{3}{5} & \frac{2}{5} & \frac{2}{5} \vspace{2mm} \\
\frac{2}{5} & \frac{2}{5} & \frac{2}{5} & -\frac{3}{5} & \frac{2}{5} \vspace{2mm} \\
\frac{2}{5} & \frac{2}{5} & \frac{2}{5} & \frac{2}{5} & -\frac{3}{5}
\end{array}
\right)
\end{equation*}
in the case $N=5$, with a similar structure of the scattering matrices $A^{\star}$ in the general case $N>5$.
\end{rem}
\subsection{A Linear Classical Solution of the Second Boundary-value Problem}
We now show that the unique linear $N$-body scattering map that conserves both momentum and energy is a classical solution of the second boundary-value problem for the Jacobian equation. As a result of this, it follows immediately that $\sigma^{\star}$ generates a billiard flow $\{T^{t}_{\star}\}_{t\in\mathbb{R}}$ with the property that $T^{t}_{\star}\#\lambda=\lambda$ for all $t\in\mathbb{R}$. In what follows, the sign adopted in the Jacobian PDE depends on the parity of the dimension $N$ with which we work.
\begin{prop}
Suppose $N\geq 3$. The map $\sigma^{\star}$ is a classical solution of the second boundary-value problem
\begin{equation*}
\left\{
\begin{array}{l}
\displaystyle \mathrm{det}(D\sigma(V))=(-1)^{N-1}\frac{H(V)}{H(\sigma(V))}\quad \text{for}\hspace{2mm}V\in(\mathcal{V}^{-})^{\circ}, \vspace{2mm} \\
\sigma(\mathcal{V}^{-})=\mathcal{V}^{+}.
\end{array}
\right.
\end{equation*}
\end{prop}
\begin{proof}
It has already been shown above that the $N$-body scattering map $\sigma^{\star}$ satisfies the range condition $\sigma^{\star}(\mathcal{V}^{-})=\mathcal{V}^{+}$. All that remains to be shown is that the smooth map $\sigma^{\star}$ satisfies the PDE
\begin{equation*}
\mathrm{det}(D\sigma^{\star}(V))=(-1)^{N-1}\frac{H(V)}{H(\sigma^{\star}(V))}    
\end{equation*}
for all $V\in(\mathcal{V}^{-})^{\circ}$. We begin by noting that
\begin{equation*}
\left(\frac{2}{N}\mathbf{1}\otimes \mathbf{1}-\sum_{k=1}^{N}e_{k}^{N}\otimes e_{k}^{N}\right)V\cdot(e_{i}^{N}-e_{j}^{N})=-v_{i}+v_{j},
\end{equation*}
whence $\sigma_{i}^{\star}(V)-\sigma_{j}^{\star}(V)=-v_{i}+v_{j}$ and in turn that
\begin{equation*}
H(\sigma^{\star}(V))=H(V)    
\end{equation*}
for all $V\in\mathcal{V}^{-}$. We also recall that the matrix $A^{\star}$ which characterises the linear map $\sigma^{\star}$ satisfies
\begin{equation*}
A^{\star}\mathbf{1}=\mathbf{1}   
\end{equation*}
as well as
\begin{equation*}
A^{\star}E_{i}^{N}=-E_{i}^{N}    
\end{equation*}
for $i\in\{1, ..., N-1\}$. As the determinant of $A^{\star}$ equals the product of its eigenvalues, it holds that $\mathrm{det}(A^{\star})=(-1)^{N-1}$ for all $N\geq 3$ and in turn that
\begin{equation*}
\mathrm{det}(D\sigma^{\star}(V))=(-1)^{N-1}    
\end{equation*}
for all $V\in\mathcal{V}^{-}$. The statement of the Proposition follows therefrom.
\end{proof}
\subsection*{Acknowledgments}
The author would like to thank Laure Saint-Raymond for an interesting discussion on invariant measures for hard sphere dynamics.
\bibliographystyle{siam}
\bibliography{bib}

\end{document}